\documentclass{article}

\usepackage{PRIMEarxiv}

\usepackage[utf8]{inputenc} 
\usepackage[T1]{fontenc}    
\usepackage[colorlinks, urlcolor=blue, citecolor=blue, linkcolor=black]{hyperref}       
\usepackage{url}            
\usepackage{booktabs}       
\usepackage{amsfonts}       
\usepackage{nicefrac}       
\usepackage{microtype}      
\usepackage{lipsum}
\usepackage{fancyhdr}       
\usepackage{appendix}
\usepackage{graphicx}       
\usepackage[justification=centerlast, indention=.5cm]{caption}
\usepackage[position=b]{subcaption}
\usepackage{amsmath}
\usepackage{comment}
\usepackage{tabularx}
\usepackage{comment}
\usepackage[ruled]{algorithm2e}
\usepackage{amsthm}
\usepackage{dsfont}
\usepackage{cleveref}

\crefformat{footnote}{#2\footnotemark[#1]#3}

\graphicspath{{media/}}     

\numberwithin{equation}{section}
\theoremstyle{plain}
\newtheorem{theorem}{Theorem}[section]
\newtheorem{corollary}[theorem]{Corollary}
\newtheorem{lemma}[theorem]{Lemma}
\newtheorem{remark}[theorem]{Remark}
\newtheorem{proposition}[theorem]{Proposition}

\newcommand{\ind}[1]{\mathds{1}_{#1}}

\title{Trade-off between predictive performance and FDR control for high-dimensional Gaussian model selection
}

\author{
  Perrine Lacroix \\
  Laboratoire de Mathématiques d’Orsay, CNRS, Université Paris-Saclay, Orsay, France \\
  Université Paris-Saclay, CNRS, INRAE, Université Evry, Institute of Plant Sciences Paris-Saclay (IPS2), \\
  91190, Gif sur Yvette, France \\
  Université Paris Cité, Institute of Plant Sciences Paris-Saclay (IPS2), 91190, Gif sur Yvette, France \\
  \url{perrine.lacroix@universite-paris-saclay.fr} \\
   \And
   Marie-Laure Martin \\
    Université Paris-Saclay, AgroParisTech, INRAE, UMR MIA Paris-Saclay, 91120, Palaiseau, France \\
    Université Paris-Saclay, CNRS, INRAE, Université Evry, Institute of Plant Sciences Paris-Saclay (IPS2), \\
    91190, Gif sur Yvette, France \\
    Université Paris Cité, Institute of Plant Sciences Paris-Saclay (IPS2),    91190, Gif sur Yvette, France \\
  \url{marie-laure.martin@inrae.fr} \\
}

\begin{document}
\maketitle

\begin{abstract}
In the context of high-dimensional Gaussian linear regression for ordered variables, we study the variable selection procedure via the minimization of the penalized least-squares criterion. We focus on model selection where the penalty function depends on an unknown multiplicative constant commonly calibrated for prediction. We propose a new proper calibration of this hyperparameter to simultaneously control predictive risk and false discovery rate. 
We obtain non-asymptotic bounds on the False Discovery Rate with respect to the hyperparameter and we provide an algorithm to calibrate it. This algorithm is based on quantities that can typically be observed in real data applications. The algorithm is validated in an extensive simulation study and is compared with several existing variable selection procedures. Finally, we study an extension of our approach to the case in which an ordering of the variables is not available.
\keywords{Ordered variable selection \and Prediction \and FDR \and High-dimension \and Gaussian regression \and Hyperparameter calibration}
\end{abstract}

\section{Introduction.}
\subsection{Problem statement.}

We consider the following high-dimensional univariate Gaussian linear regression model~:
\begin{equation}
    Y = X \beta^* + \varepsilon.
    \label{regression}
\end{equation}
The random response vector $Y=\Big((Y_i)_{\{1 \leq i \leq n\}}\Big)^T \in \mathbb{R}^n$ is regressed on $p$ deterministic vectors~: $X_1 = \Big((x_{i1})_{\{1 \leq i \leq n\}}\Big)^T,\cdots,X_p=\Big((x_{ip})_{\{1 \leq i \leq n\}}\Big)^T$. The design matrix of size $n \times p$ is denoted by $X=\left(X_1,\cdots,X_p\right)$. The noise $\varepsilon = \Big((\varepsilon_i)_{\{1 \leq i \leq n\}}\Big)^T$ is assumed to be Gaussian~: $\varepsilon \sim \mathcal{N}(0,\sigma^2 I_n)$ with $\sigma^2>0$. 
In the high-dimensional context, additional assumptions of regularity are required and we assume that $\beta^*$ is sparse, meaning that only a few coefficients are non-zero. 
In the following, a variable $X_j$ corresponding to a non-zero coefficient $\beta^*_j$ is called an active variable. Otherwise the variable is said to be non-active. \\
In this paper, we are interested in variable selection.
We refer the reader to \cite{hastie2009elements} and references therein. To the best of our knowledge, some variable selection procedures focus on the prediction of the response variable $Y$ through a control of the predictive risk. Others focus on limiting the number of selected non-active variables through a control of the False Discovery Rate. There also exists procedures where several cost functions are considered simultaneously. In the line of the latter, our goal is to identify a set of variables from a model selection procedure by limiting the selection of non-active variables while maintaining accurate prediction performances.

\subsection{Related works.}
\label{previously}
In a variable selection procedure, a cost function has to be defined.
The predictive risk (PR) and the False Discovery Rate (FDR) are the common used cost functions. 

\vspace{-0.2 cm}
\paragraph{The penalized methods to control the predictive risk.}
The penalization procedure balances goodness of fit and sparsity~: the smaller the penalty function, the better the fitting to the data but the higher the number of selected variables.
In a high-dimensional setup, the most popular method is the Lasso criterion \cite{tibshirani1996regression} where the estimator $\Hat{\beta}_\lambda$ of $\beta^*$ is the solution of~:
\begin{equation}
    \Hat{\beta}_\lambda = \underset{\beta \in \mathbb{R}^p}{\arg\min} \Big\{||Y-X\beta||_2^2 + \lambda |\beta|_1 \Big\},
    \label{lasso_eq}
\end{equation}
where $|\cdot|_1$ and $||\cdot||_2$ denote the $\ell_1$-norm and the euclidean norm of a vector respectively. 
The main challenge is to calibrate the hyperparameter $\lambda >0$.
If $\lambda$ is chosen to be proportional to $\sigma \sqrt{\frac{\log(p)}{n}}$, then the predictive risk is bounded \cite{bunea2007aggregation,bunea2007sparsity}. However, the noise being usually unknown, the choice of $\lambda$ remains tricky. Therefore, an alternative is to solve the Lasso criterion for a $\lambda$ within a reasonable interval by using subsamples \cite{meinshausen2010stability} or resamples \cite{bach2008bolasso}. The selected variables are then defined as the variables with the highest selection frequencies. Such alternative is no longer sensitive to the choice of $\lambda$ but the main challenge lies in the threshold on the frequency defining the selected variables. \\
In this paper, we consider a model selection procedure composed of three steps. The first step consists of solving the Lasso criterion on a relevant grid $\Lambda$. Each $\lambda \in \Lambda$ defines a variable subset to get a collection $\mathcal{M}$ of relevant subsets of variables with a wide range of sizes.
In the second step, the least-squares estimator onto each variable subset of $\mathcal{M}$ is calculated leading to a collection of estimators $\left(\Hat{\beta}_m\right)_{m \in \mathcal{M}}$. Lastly, the following penalized least-squares minimization is solved to select the best $m$ of $\mathcal{M}$~:
\begin{equation}
    \Hat{m} = \underset{m \in \mathcal{M}}{\arg\min} \Big\{ ||Y - X \Hat{\beta}_m||_2^2 + \text{pen}(D_m) \Big\},
    \label{penalization_LS}
\end{equation}
where $D_m$ is the dimension of the model $m$ and the function $\textit{pen}$ is a penalty function increasing with $D_m$. \\
Selecting $\Hat{m}$ from $\mathcal{M}$ by minimizing~(\ref{penalization_LS}) corresponds to selecting $\Hat{\lambda}$ from $\Lambda$ by minimizing~(\ref{lasso_eq}). Hence, the main challenge is the definition of $\textit{pen}$ that achieves an optimal trade-off between goodness of fit and sparsity within $\mathcal{M}$. Popular methods of model selection include $V-$fold cross-validation \cite{geisser1975predictive,arlot2010survey}, AIC \cite{akaike1973information}, Cp-Mallows \cite{mallows2000some}, BIC \cite{schwarz1978estimating} and eBIC \cite{chen2008extended}. For these penalty functions, the predictive risk is bounded when $\sigma^2$ is known and when the sample size $n$ tends to infinity. When $n$ is fixed, relatively small, and possibly smaller than the dimension $p$, a non-asymptotic point of view is preferable to get properties for all couples of $(n,p)$. In this direction, \cite{birge2007minimal} propose some penalty functions depending on the collection complexity such that $\Hat{m}$ guarantees non-asymptotic optimal control of the predictive risk. If the model collection is nested with a known variance, $\text{pen}(D_m) = 2 \sigma^2 D_m$ allows to achieve an optimal non-asymptotic control of the predictive risk \cite{akaike1973information}. If the model collection is fixed and large (for instance with an exponential growth with $D_m$) and if the variance is unknown, this optimal control is obtained with data-driven penalties \cite{birge2007minimal,baudry2012slope}. Lastly, if the model collection is data-dependent and if the variance $\sigma^2$ is unknown, the LinSelect penalty \cite{baraud2009gaussian,giraud2012high} guarantees an optimal control of the predictive risk. 

\vspace{-0.2 cm}
\paragraph{Multiple testing methods to control the False Discovery Rate.}
In the multiple testing procedure, the $p$ tests $H_0 = \{\beta^*_j = 0\}$ versus $H_1 = \{\beta^*_j \neq 0\}$ are performed independently to get a list of $p$-values.
Variables associated with a $p$-value smaller than a threshold are selected and the challenge is to find this threshold to obtain an upper bound on a function of the number of selected non-active variables. 
Several methods control the Family-Wise Error (FWER) which is the probability of selecting at least one non-active variable \cite{bonferroni1936teoria,simes1986improved}. However, these methods tend to be conservative, leading to a tiny set of selected variables. 
An alternative is to control the FDR which is the expectation of the proportion of non-active variables among the selected ones. The authors of \cite{benjamini1995controlling} first provide a threshold assuming independence of the $p$-values. 
This hypothesis is then relaxed in \cite{benjamini2001control,storey2004strong,romano2008control,leung2021zap}. \\
Instead of considering the $p$-values, the knockoff filter method \cite{barber2015controlling} introduces copies of the columns of $X$ constructed to
function as non-active variables to calibrate a threshold on test statistics.

\vspace{-0.2 cm}
\paragraph{The simultaneous control of several cost functions.}
Controlling PR or FDR is commonly performed independently in the literature and yield different sets of selected variables. For a PR control, selected variables aim at correctly predicting a new observation of $Y$, without guaranteeing that the set of selected variables does not contain non-active variables. Conversely, when the cost function is the FDR, the number of non-active variables is controlled at the price that some active variables are not selected.\\
Therefore, recent works have been proposed to combine prediction and FDR approaches to select all active variables without selecting non-active ones. For instance, \cite{zhou2009thresholding} propose a multi-step algorithm where a threshold procedure is applied to some Lasso estimators computed for specific values of $\lambda$.
In addition to prediction performances, a consistency property on the selected variable set is satisfied under some conditions on $X$. Another idea is the post-selection inference \cite{berk2013valid,lee2016exact} where the principle is to test the relevance of each selected variable by a model selection procedure. Valid confidence intervals are provided from conditional hypothesis tests for each model of the collection in addition to a PR control. Their work has been generalized by  \cite{hyun2018exact,chen2021powerful,duy2021more} and a review can be found in \cite{zhang2022post}. \\
In a completely different direction, \cite{genovese2002operating, genovese2004stochastic} propose to control the False Negative Rate (FNR) in addition to the FDR. A good FNR control ensures that most of the active variables are selected. So, minimizing a weighted sum of FDR and FNR provide a set of variables close to the set of active variables. 
However, improving FDR control deteriorates FNR control and vice versa. Hence, optimal controls of both criteria are impossible to achieve. \\
Some other papers propose to combine the FDR with the PR. Additional motivation to consider the PR is its behavior between the learning phase and the over-fitting phase. In the learning phase, the addition of a variable in the selected set drastically reduces the PR, whereas in the over-fitting phase, it increases proportionally to the noise level.
Firstly, in the standard multivariate normal mean problem with a known variance, \cite{abramovich2006adapting} propose a penalty function in the model selection procedure built from a multiple testing procedure. They obtain simultaneously sharp asymptotic bounds of the FDR and the PR. Then, \cite{bogdan2013statistical} propose the Sorted $\ell_1$ penalized estimator (SLOPE)
which is the minimizer of the Lasso criterion~(\ref{lasso_eq}) where $\lambda$ is replaced by a $p$-vector built from a multiple testing procedure. For the orthogonal design, their approach achieves a non-asymptotic control of the FDR and satisfies a minimum value of the total mean squared error with minimax convergence rate \cite{su2016slope}. This asymptotic convergence of the FDR has been generalized under a wide range of hypotheses, for instance, for a random design in \cite{kos2020asymptotic}.

\vspace{-0.2 cm}
\paragraph{{Ordered variable selection.}}
The ordered variable framework has attracted much attention recently, especially to address high-dimensional
problems.
In literature, a large class of methods exists dealing with variables having a natural ranking~: \cite{bickel2008regularized} for the regression framework, \cite{levina2008sparse} with the nested lasso penalty and \cite{huang2006covariance} for covariance matrix estimation. This assumption allows for drastically reducing the estimation complexity. 
We develop our theoretical framework under this assumption. It can be applied to data sets where the
assumption that an ordering of the variables is available a priori (e.g., from suitable forms of prior knowledge) is appropriate. \\
However, in most applications,
no canonical ranking of the variables is available and having a natural order on variables becomes a strong assumption.
In this case, alternatives consist of proposing a candidate order from random procedures and applying theoretical statistical methods on top of these random variable rankings. Several approaches have been proposed in the literature to provide such rankings. The most used ones are based either on a regularization path which is built with the Lasso type equation solving \cite{tibshirani1996regression} or on a decision tree 
\cite{li1984classification}. However, these approaches suffer from instability in that a small modification of the initial sample could radically change the variable order \cite{kalousis2007stability}. 
To circumvent this instability problem, one solution is to add a sampling 
procedure like the bootstrap \cite{breiman2001random}.
We adopt this point of view in this article to generalize our theoretical results to non-ordered variable selection.

\subsection{Main contributions.}
The originality of this paper is to obtain a control of the FDR in addition to the PR control in model selection through a convenient calibration of the penalty. \\
We assume variables are ranked according to their importance for the response variable $Y$; $X_1$ being the most important one, $X_2$ being the second one, $\cdots$, and $X_p$ being the least important one. In Gaussian linear regression, the order is given by the partial correlation between $Y$ and each $X_j$.
A natural model collection is the one containing the nested models respecting the variable order. This framework sounds restrictive but allows to derive theoretical expressions of the FDR in the considered model selection procedure. 
According to \cite{birge2007minimal}, all the penalty functions defined by~: 
\begin{equation}
 	\text{pen}(D_m) = K \sigma^2 D_m, \quad \forall m \in \mathcal{M},
 	\label{our_penalty_FDR}
\end{equation}
provide a non-asymptotic control of the PR for $K>1$ when variables are ranked.

\vspace{0.1 cm}
\textbf{Theoretical bounds on the FDR in model selection~:} Although the model selection procedure is built for a PR control, we obtain non-asymptotic lower and upper bounds on the FDR with respect to $K>0$ when $\sigma^2$ is known. We show that these bounds only involve some evaluations of cumulative distribution functions of the standard Gaussian and of some chi-squared variables. Whatever the noise level, FDR is always strictly positive. When $K$ tends to infinity, the FDR converges to $0$ with an exponential rate. So, a low value of the FDR is satisfactory as soon as the value of $K$ is not too large. 

\vspace{0.1 cm}
\textbf{Calibration of the hyperparameter $K$~:} The obtained theoretical bounds depend on the parameters $\beta^*$ and $\sigma^2$. We replace them with estimators to obtain completely data-dependent bounds on the FDR. 
Then, we propose a calibration of the hyperparameter $K$ to control a trade-off between FDR and PR. Our algorithm is validated on an extensive simulation study and is compared with several existing variable selection procedures.

\vspace{0.1 cm}
\textbf{Towards a non-ordering variable selection~:} 
From a practical point of view, a crucial assumption of this work is the knowledge of the variable ranking. We investigate empirically an extension in which the variable ordering is not given beforehand but estimated using a data-driven procedure to build
random model collections.

\subsection{Outline of the paper.}
The rest of the paper is organized as follows. Section~\ref{Model_notation} introduces the Gaussian linear regression model and some notations. Section~\ref{main_results} contains theoretical results. 
Since an increase of the hyperparameter $K$ leads to a decrease of the FDR, it motivates the study of the FDR function in model selection with respect to $K$. As the FDR has an intractable expression, bounds are obtained when the variable order and the variance are known. We establish an exponential convergence rate of the FDR function when $K$ tends to infinity. The special case of orthogonal design matrix is studied to illustrate the main results.  
In Section~\ref{trade_off_fixed_collection}, 
an algorithm is proposed to calibrate the hyperparameter $K$ in the penalty function to achieve a suitable trade-off between FDR and PR controls. 
It is based on simultaneous evaluations of the prediction performance and the FDR of the models, which are calculated from properly chosen estimators of $\sigma^2$ and $\beta^*$.
We then present a 
study to generalize our procedure  to non-ordered variable selection and we compare our algorithm with some existing variable selection procedures. 
Section~\ref{conclusion_FDR} contains a conclusion and a discussion of prospective work. In Section~\ref{proofs}, proofs of all the theoretical results are provided.
Lastly, the simulation protocol considered in this paper is described in Section~\ref{simulation_protocol}.

\section{Model and notations.}
\label{Model_notation}
Let us consider the Gaussian linear regression model given in~(\ref{regression}). 
We define $q = \min(n,p)$ and assume that $\left(X_1,\cdots,X_q\right)$ is a family of linearly independent vectors. We consider the deterministic and nested model collection of linear spaces~:
\begin{equation}
    \mathcal{M} = \Big\{m_0=\{0\}, m_1=\text{Span}(X_1), \cdots , m_q=\text{Span}(X_1,X_2,\cdots,X_q)\Big\}.
    \label{nested_collection}
\end{equation}
By construction, the true model $m^* = \text{Span}\big(X_j, \ j \ \text{s.t.} \ \beta^*_j \neq 0\big)$ belongs to $\mathcal{M}$. \\
For each $m \in \mathcal{M}$, $D_m$ is the dimension of $m$ and $\Hat{\beta}_m$ is the least-squares estimator onto $m$~: 
\begin{equation*}
     \Hat{\beta}_m = \underset{\{\beta, X\beta \in m\}}{\arg\min} \Big\{ ||Y-X\beta||_2^2 \Big\}.
\end{equation*}
With the definition of $q$ and properties on the family $\left(X_1,\cdots,X_q\right)$, $\Hat{\beta}_m$ is unique for each $m \in \mathcal{M}$. \\
For all $K>0$, we define the function $\text{crit}_K$ on $\mathcal{M}$ as~:
$$ \text{crit}_K(m) =  ||Y-X\Hat{\beta}_m||_2^2 + K \sigma^2 D_m, $$
and the selected model $\Hat{m}(K)$ by~:
\begin{equation}
    \Hat{m}(K) = \underset{m \in \mathcal{M}}{\arg\min} \Big\{ \text{crit}_K(m) \Big\}.
    \label{our_model_selection}
\end{equation}
We define $\text{PR}(m)$ the predictive risk associated to the model $m \in \mathcal{M}$ by~:
\begin{equation}
    \text{PR}(m) = \mathbb{E}\Big[||Y-X\Hat{\beta}_m||_2^2\Big],
    \label{theorical_rik}
\end{equation}
where $\mathbb{E}$ denotes the expectation under the distribution of $Y$ satisfying~(\ref{regression}). We define successively $FP(m)$ the number of variables contained in $m$ but not in $m^*$, the false discovery proportion by~:
$$ \text{FDP}(m) = \frac{\text{FP}(m)}{\max(D_m,1)};$$
and the False Discovery Rate by~: 
$$ \text{FDR}(m) = \mathbb{E}\Big[\text{FDP}(m)\Big],$$
where $\mathbb{E}$ still denotes the expectation under the distribution of Y satisfying~(\ref{regression}), so that $\text{FDR}(m)$ is deterministic even in the case where $m$ is random. 

Finally, 
the notation $\langle . , . \rangle$ denotes the canonical scalar product in $\mathbb{R}^n$, $\Pi_{\mathcal{X}}$ denotes the orthogonal projection function onto the space $\mathcal{X}$, $\Phi$ denotes the standard Gaussian cumulative distribution function and $F_{\chi^2(k)}$ is the cumulative distribution function of a chi-squared variable with $k$ degrees of freedom. By convention, an intersection or an union from indices $k$ to $\ell$ with $k>\ell$ are the intersection or the union over an empty set. In the same way, the set $\{k,\cdots,\ell\}$ is empty if $k>\ell$.

\section{Main results.}
\label{main_results}
In this section, the variance $\sigma^2$ is supposed to be known.
We first present intuitions that lead to study $\text{FDR}(\Hat{m}(K))$ in model selection. 
Non-asymptotic bounds on $\text{FDR}(\Hat{m}(K))$ are obtained in Theorem~\ref{theorem}, as well as asymptotic behaviors when $K$ tends to infinity in Corollary~\ref{asympt_K}. Finally, the particular case where $X$ is the orthogonal design matrix is studied to illustrate the main results.

\subsection{Intuitions.}
\label{key_ideas}
According to \cite{birge2007minimal}, the penalty function~(\ref{our_penalty_FDR}) satisfies a non-asymptotic control of the PR if and only if $K>1$. 
The constant $K=2$ allows to achieve the optimal asymptotic control of the PR. Hence, $2$ is commonly chosen in practice but other values of $K$ close to $2$ can give identical if not better non-asymptotic prediction performances. In this direction, we propose to calibrate the hyperparameter $K$ among those leading to prediction performances close to or better than for $K=2$ while satisfying a control of the FDR. The calibration is based on both $\text{PR}(\Hat{m}(K))$ and $\text{FDR}(\Hat{m}(K))$ functions with respect to $K$. \\
\begin{figure}[!h]
    \centering
    \includegraphics[width=1\linewidth]{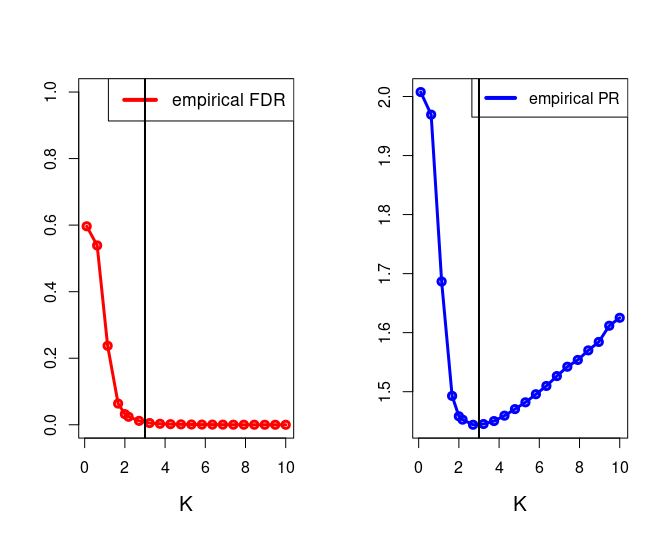}
    \caption{Curves of the empirical values of $\text{FDR}\big(\Hat{m}(K)\big)$ and $\text{PR}\big(\Hat{m}(K)\big)$ for the \textit{toy data set} described in Section~\ref{simulation_protocol}. The vertical lines correspond to $K=3$.}
    \label{FDR_RP10}
\end{figure}

\begin{table}[ht]
\centering
\begin{tabular}{|r|r|r|r|r|r|r|r|r|r|r|r|}
  \hline
   K  & 0.1 & 1 & 2 & 3 & 4 & 5 & 6 & 7 & 8 & 9 & 10 \\ 
  \hline
  \hline
  empirical & & & & & & & & & & & \\
    $\text{FDR}(\Hat{m}(K))$ & 0.80 & 0.38 & 0.05 & 0.01 & 0.00 & 0.00 & 0.00 & 0.00 & 0.00 & 0.00 & 0.00 \\ 
    \hline
  empirical & & & & & & & & & & & \\
  $\text{PR}(\Hat{m}(K))$ & 2.01 & 1.55 & 1.25 & 1.24 & 1.25 & 1.26 & 1.28 & 1.30 & 1.32 & 1.33 & 1.36 \\ 
   \hline
\end{tabular}
\caption{Values of the empirical estimators of $\text{PR}(\Hat{m}(K))$ and $\text{FDR}(\Hat{m}(K))$ according to $K$ for the \textit{toy data set} described in Section~\ref{simulation_protocol}.}
\label{table_intuition}
\end{table}
\noindent We propose an example to illustrate our point and intuition. 
In Figure~\ref{FDR_RP10}, we plot the empirical estimators of $\text{PR}(\Hat{m}(K))$ and $\text{FDR}(\Hat{m}(K))$ on a regular grid of positive values of $K$. Graphs are obtained from the \textit{toy data set} described in Section~\ref{simulation_protocol} and values are transferred to Table~\ref{table_intuition}. 
We observe that the empirical values of $\text{PR}(\Hat{m}(K))$ are kept low for $K\in [2,4]$ while the $\text{FDR}(\Hat{m}(K))$ function decreases with $K$ until $0$. Here, the choice $K=3$ is more judicious than $K=2$: it ensures a stronger and positive control of the FDR while satisfying similar prediction performances (a FDR of zero is not relevant as it means that no variable is selected). We also observe that while FDR decreases with $K$, PR increases from a certain value of $K \geq 2$. To control PR and FDR simultaneously, the constant $K$ must be close to $2$.\\ 
Increasing the constant $K$ to limit the non-active variable selection is known for the asymptotic point of view. Indeed, AIC and Cp-Mallows penalties \cite{akaike1973information,mallows2000some}, where $K$ equals $2$, give asymptotically the best set of variables for prediction performances; while BIC penalty \cite{schwarz1978estimating}, where $K$ is fixed to $\log(n)$, exactly recovers asymptotically the set of active variables. Obtaining the asymptotic properties of AIC, Cp-Mallows and BIC penalties simultaneously is impossible \cite{yang2005can}, but it suggests that a value of $K \in [2,\log(n)]$ would get reasonable (but not necessarily optimal) values for both PR and FDR in a non-asymptotic framework. In this way, we propose to study the function $\text{FDR}\big(\Hat{m}(K)\big)$ in the model selection procedure~(\ref{our_model_selection}) where the penalty function is~(\ref{our_penalty_FDR}) in the ordered variable setting.

\subsection{Bounds on the FDR in model selection.}

\subsubsection{FDR expression in model selection for ordered variables.}
Let us assume that $K>0$ and $\text{crit}_K$ is injective on $\mathcal{M}$.
If $D_m^* = q$, $\text{FDR}(\Hat{m}(K)) =~0$. Otherwise, the $\text{FDR}(\Hat{m}(K))$ is expressed within the model selection procedure as~: 
\begin{equation}
    \text{FDR}(\Hat{m}(K)) = \underset{r = D_{m^*}+1}{\overset{q}{\sum}} \frac{r-D_{m^*}}{r} \mathbb{P}\left(\left\{\underset{\underset{\ell \neq r}{\ell = 0}}{\overset{q}{\cap}} \{ \text{crit}_K(m_r) < \text{crit}_K(m_\ell)\}\right\}\right).
    \label{FDR_with_crit}
\end{equation} 

\vspace{0.7 cm}

A detailed proof of~(\ref{FDR_with_crit}) can be found in Subsection~\ref{proof_FDR_expression}. \\

\noindent By using the decomposition 
$$\left\{\underset{\ell = 0}{\overset{r-1}{\cap}} \{ \text{crit}_K(m_r) < \text{crit}_K(m_\ell)\}\right\} \bigcap \left\{\underset{\ell = r+1}{\overset{q}{\cap}} \{ \text{crit}_K(m_r) < \text{crit}_K(m_\ell)\}\right\}$$
of the term
$\underset{\underset{\ell \neq r}{\ell = 0}}{\overset{q}{\cap}} \{ \text{crit}_K(m_r) < \text{crit}_K(m_\ell)\} ,$
we obtain the following proposition~:  

\vspace{0.5 cm}
\begin{proposition}
Let us consider the ordered variable framework and the model collection~(\ref{nested_collection}) where $q = \min(n,p)$, $m^* \in \mathcal{M}$ and $D_m^*<q$. Let us assume that $\text{crit}_K$ is injective on $\mathcal{M}$. Let $\left(u_1,\cdots,u_n\right)$ be an orthonormal basis of $\mathbb{R}^n$ such that $\text{Span}(X_1,\cdots,X_j) = \text{Span}(u_1,\cdots,u_j), \ \forall j \in \{1,\cdots,q\}$.
\\
Then, $\forall K>0$, 
\begin{equation}
    \text{FDR}(\Hat{m}(K)) = \underset{r = D_{m^*}+1}{\overset{q}{\sum}} \frac{r-D_{m^*}}{r} \ P_r(K) \ Q_r(K,\beta^*,\sigma^2),
    \label{FDR_1}
\end{equation}

\vspace{0.3 cm}
where for each $r \in \{D_{m^*}+1, \cdots,q\}$,
\begin{equation}
    P_r(K) = \mathbb{P}\bigg( \underset{\ell = r+1}{\overset{q}{\cap}} \Big\{ \underset{k=r+1}{\overset{\ell}{\sum}} Z_k^2 < K(\ell-r) \Big\} \bigg), 
    \label{second_term}
\end{equation}
\hspace{4 cm} $\text{where} \ \ Z_k \overset{i.i.d.}{\sim} \mathcal{N}(0,1), \ \ \forall k \in \{r+1,\cdots,q\}$,
\begin{equation*}
    \text{and} \ Q_r(K,\beta^*,\sigma^2) = \mathbb{P}\bigg( \underset{\ell = 0}{\overset{r-1}{\cap}} \Big\{ \underset{k=\ell+1}{\overset{r}{\sum}} \langle Y, u_k \rangle^2 > K\sigma^2(r-\ell) \Big\} \bigg).
\end{equation*}
\label{Proposition}
\end{proposition}

A proof of Proposition~\ref{Proposition} can be found in Subsection~\ref{proof_FDR_expression}. \\ 

\subsubsection{General bounds.}
In~(\ref{FDR_1}), the $P_r(K)$ terms do not depend on data. Conversely, the $Q_r(K,\beta^*,\sigma^2)$ terms depend on the data. Thus, to understand the behavior of the $\text{FDR}$ function with respect to $\Hat{m}(K)$, we propose to bound the $Q_r(K,\beta^*,\sigma^2)$ terms in the following theorem~:

\begin{theorem}
Let us consider the ordered variable framework and the model collection~(\ref{nested_collection}) where $q = \min(n,p)$. Let us suppose that $m^* \in \mathcal{M}$ and $D_m^*<q$. The notation $\Phi$ stands for the standard Gaussian cumulative distribution function and $F_{\chi^2(k)}$ is the cumulative distribution function of a chi-squared variable with $k$ degrees of freedom. Let us assume that $\forall K>0, \text{crit}_K$ is injective on $\mathcal{M}$. Let $\left(u_1,\cdots,u_n\right)$ be an orthonormal basis of $\mathbb{R}^n$ such that $\text{Span}(X_1,\cdots,X_j) = \text{Span}(u_1,\cdots,u_j), \ \forall j \in \{1,\cdots,q\}$.
 \\
Then, $\forall K>0$, $\Hat{m}(K)$ satisfies~: 

\vspace{0.1 cm}
\begin{equation}
    b(K,\beta^*,\sigma^2) \leq \text{FDR}(\Hat{m}(K)) \leq B(K,\beta^*,\sigma^2),
    \label{FDR_bounds}
\end{equation}

\vspace{0.1 cm}
where $\Big[K \mapsto b(K,\beta^*,\sigma^2) \Big]$ and $\Big[K \mapsto B(K,\beta^*,\sigma^2)\Big]$ are two real-valued functions on $\mathbb{R}^{+}$ defined by~: 
\begin{align}
     &b(K,\beta^*,\sigma^2) = \sum\limits_{r = D_{m^*}+1}^q \Bigg(\frac{r-D_{m^*}}{r} P_r(K) \ \underline{f}_r(K,\beta^*,\sigma^2) \Bigg), \notag \\
    & B(K,\beta^*,\sigma^2) = \sum\limits_{r = D_{m^*}+1}^q \Bigg( \frac{r-D_{m^*}}{r} P_r(K) \ \overline{f}_r(K,\beta^*,\sigma^2)\Bigg),
    \label{small_large_term}
\end{align}
where for all $K>0$, $P_r(K)$ is defined in~(\ref{second_term}) and
\begin{enumerate}

    \item $\text{for each} \ r \in \{D_{m^*}+1, \cdots,q\} \ \text{and for all} \ \ell \in \{1, \cdots,r\}, \ \underline{f}_{\ell}(\cdot,\beta^*,\sigma^2)$ \ is 
    defined by~:
    \begin{align*}
        \underline{f}_1(K,\beta^*,\sigma^2) &= G_1 \\
        \underline{f}_{\ell}(K,\beta^*,\sigma^2) &= G_{\ell} + H_{\ell} \ \underline{f}_{\ell-1}(K,\beta^*,\sigma^2), \quad  \forall \ell \in \{2,\cdots,r\},
    \end{align*}
     $\hspace{0.2 cm} \text{with for} \ \ \ell \in \{1, \cdots, D_{m^*}\}$~: 
    \begin{align*}
        &G_{\ell} = 2 - \bigg( \Phi\big(\sqrt{\ell K} - \frac{\langle X \beta^*, u_\ell \rangle}{\sigma}\big) + \Phi\big(\sqrt{\ell K} + \frac{\langle X \beta^*, u_\ell \rangle}{\sigma} \big)\bigg),
    \end{align*}
    $\hspace{0.8 cm}$ for $\ell \in \{2, \cdots, D_{m^*}\}$~: 
    \begin{align*}
        &H_{\ell} = \Phi\Big(\sqrt{\ell K}-\frac{\langle X \beta^*, u_\ell \rangle}{\sigma}\Big) + \Phi\Big(\sqrt{\ell K}+ \frac{\langle X \beta^*, u_\ell \rangle}{\sigma}\Big) \\
        &\quad \quad \quad - \bigg(\Phi\Big(\sqrt{K}-\frac{\langle X \beta^*, u_\ell \rangle}{\sigma}\Big) + \Phi\Big(\sqrt{K}+\frac{\langle X \beta^*, u_\ell \rangle}{\sigma}\Big)\bigg), 
    \end{align*}
    $\hspace{0.9 cm} \text{for} \ \ \ell \in \{D_{m^*}+1, \cdots, r\}~:$
    \begin{align*}
         &G_{\ell} = 2 \bigg(1-\Phi\big(\sqrt{\ell K}\big)\bigg) \\
        &H_{\ell} = 2 \bigg( \Phi\big(\sqrt{\ell K}\big) - \Phi\big(\sqrt{K}\big)\bigg),
    \end{align*}
    \item $ \forall r \in \{D_{m^*}+1, \cdots,q\}, \ \overline{f}_{r}(\cdot,\beta^*,\sigma^2)$ is 
    defined by~: 
    \begin{align*}
        \overline{f}_r(K,\beta^*,\sigma^2) &= 1-\max\Bigg(\underset{\ell \in \{1,\cdots,r-D_{m^*}\}}{\max}\Big(F_{\chi^2(\ell)}(\ell K)\Big), \\
        & \hspace{0.2 cm} \underset{\ell \in \{r-D_{m^*}+1,\cdots,r\}}{\max}\bigg(F_{\chi^2(\ell)}\Big(\frac{\ell K}{2} - \underset{k=r-\ell+1}{\overset{D_{m^*}}{\sum}}\frac{\langle X\beta^*,u_k \rangle^2}{\sigma^2}\Big)\bigg)\Bigg).
    \end{align*}
\end{enumerate}
\label{theorem}
\end{theorem}

\vspace{0.3 cm}
A proof of Theorem~\ref{theorem} can be found in Subsection~\ref{proof_general_bounds}. \\
Hence, although the model selection procedure is built for prediction performances, bounds on the FDR are derived with respect to $\Hat{m}(K)$. 
Terms $\underline{f}_r(K,\beta^*,\sigma^2)$ and $\overline{f}_r(K,\beta^*,\sigma^2)$ only involve evaluations of cumulative distribution functions of the standard Gaussian and chi-squared variables. So, they have a fully explicit form which simplifies the understanding of the behavior of the FDR in model selection. 
However, they depend on the unknown parameters $\beta^*$ and $\sigma^2$ for which estimators are proposed in Section \ref{parameter_estimations}.

\subsubsection{Strictly positive $\text{FDR}$.}
The following corollary gives a lower bound on the FDR independent from $\sigma^2$.
\begin{corollary}
Under the assumptions and definitions of Theorem~\ref{theorem}, $\forall K>0$~:
\begin{equation*}
    \text{FDR}(\Hat{m}(K)) \geq  \sum\limits_{r = D_{m^*}+1}^q \Bigg(\frac{r-D_{m^*}}{r} P_r(K) \ \frac{2\sqrt{2}}{\sqrt{\pi}\Big(\sqrt{rK}+\sqrt{rK+4}\Big)} e^{-\frac{rK}{2}} \Bigg) \ > 0.
\end{equation*}
\label{positive_FDR}
\end{corollary}
A proof of Corollary~\ref{positive_FDR} can be found in Subsection~\ref{proof_no_noise}. \\
From Corollary~\ref{positive_FDR}, $\text{FDR}(\Hat{m}(K))$ is always strictly positive, whatever the values of $K>0$ and $\sigma^2$.

\subsubsection{Asymptotic analysis.}
\label{Asymptotic_behaviors}
The following corollary gives the asymptotic behavior of the FDR function in model selection when $K$ tends to infinity.

\begin{corollary}
Under the assumptions and the definitions of Theorem~\ref{theorem}, the $\text{FDR}(\Hat{m}(K))$ function tends to $0$ when $K$ tends to infinity and satisfies $\forall \eta >0$,
\begin{equation}
    \text{FDR}(\Hat{m}(K)) = \underset{K \longrightarrow + \infty}{o} \Big(e^{-K(\frac{1}{2}-\eta)}\Big). \\
    \label{CV_rate_K_infinity}
\end{equation}

Furthermore, $\ \forall \eta>0, \ \exists C_\eta >0, \ \exists L_\eta >0, \ \forall K > L_\eta,$ we have~:
\begin{equation}
    \text{FDR}(\Hat{m}(K)) \geq C_\eta e^{-K\Big(\frac{D_{m^*}+1+2\eta}{2}\Big)}.
    \label{LB_rate_K_infinity}
\end{equation}

So, $\forall \varepsilon >0,$
\begin{align}
     -\frac{D_{m^*}}{2} - \frac{1}{2} - \varepsilon \ \leq \ &\underset{K \longrightarrow + \infty}{\lim\inf} \frac{1}{K}\log\Big(\text{FDR}(\Hat{m}(K))\Big) \notag \\ 
    & \underset{K \longrightarrow + \infty}{\lim\sup} \frac{1}{K}\log\Big(\text{FDR}(\Hat{m}(K))\Big)\ \leq \ - \frac{1}{2} + \varepsilon.
    \label{large_deviation}
\end{align}
\label{asympt_K}
\end{corollary}

\vspace{0.3 cm}
A proof of Corollary~\ref{asympt_K} can be found in Subsection~\ref{proof_asymptotic}.\\
From Equation~(\ref{CV_rate_K_infinity}), $\text{FDR}(\Hat{m}(K))$ tends to $0$ when $K$ tends to $+\infty$ with at least an exponential convergence rate and Equation~(\ref{LB_rate_K_infinity}) suggests that the exponential convergence rate is optimal. 

\begin{remark}
With no signal ($\beta^* = 0$ and $D_{m^*} = 0$), the asymptotic bounds in~(\ref{large_deviation}) are $- \frac{1}{2} - \varepsilon $ and $- \frac{1}{2} + \varepsilon $ and consequently~: 
$$\log\Big(\text{FDR}(\Hat{m}(K))\Big) \underset{K \rightarrow + \infty}{\sim} -\frac{1}{2} K.$$
\end{remark}

\begin{remark}
The asymptotic upper and lower bounds~(\ref{CV_rate_K_infinity}) and~(\ref{LB_rate_K_infinity}) are satisfied whatever the value of $\sigma^2>0$. It is possible to obtain the following sharpest asymptotic upper bound~: $\ \forall \Tilde{\eta} >0,$
\begin{equation}
     \text{FDR}(\Hat{m}(K)) = o\bigg(e^{- \Big(K\frac{(D_{m^*}+1-\Tilde{\eta})}{4} - \frac{1}{2\sigma^2} \underset{k=1}{\overset{D_{m^*}}{\sum}}\langle X\beta^*,u_k \rangle^2\Big)}\bigg)
    \label{CV_rate_K_infinity_more_complex}
\end{equation}
in the asymptotic regime where $K \longrightarrow + \infty$ and $\sigma \longrightarrow 0 \ \text{with} \ \frac{1}{\sigma} = \underset{\sigma \longrightarrow 0}{o}(\sqrt{K})$. 
The reader can find a proof in Subsection~\ref{proof_asymptotic}.
\label{remark_finest}
\end{remark}

\vspace{0.3 cm}

\vspace{0.3 cm}
\subsection{Illustrations of the main result in the orthogonal case.}
\label{Illustration_sub}
We propose to analyze the particular case where the design matrix $X$ is orthogonal since it leads to simplified forms for the FDR bounds that are easy to calculate.
\begin{corollary}[Application on the orthogonal case]
Under assumptions of Theorem~\ref{theorem} and by assuming that $\left(X_1,\cdots,X_q\right)$ are orthonormal with respect to $\langle . , . \rangle$, then, $\forall K>0$, $\text{FDR}(\Hat{m}(K))$ satisfies the same inequalities as (\ref{FDR_bounds}) where~: \\
for $\ell \in \{1, \cdots, D_{m^*}\}~:$
    $$G_{\ell} = 2 - \bigg( \Phi\Big(\sqrt{\ell K}-\frac{\beta^*_\ell}{\sigma}\Big) + \Phi\Big(\sqrt{\ell K}+\frac{\beta^*_\ell}{\sigma}\Big) \bigg),$$
for $\ell \in \{2, \cdots, D_{m^*}\}~:$
    $$H_{\ell} = \Phi\Big(\sqrt{\ell K}-\frac{\beta^*_\ell}{\sigma}\Big) + \Phi\Big(\sqrt{\ell K}+ \frac{\beta^*_\ell}{\sigma}\Big) - \bigg(\Phi\Big(\sqrt{K}-\frac{\beta^*_\ell}{\sigma}\Big) + \Phi\Big(\sqrt{K}+\frac{\beta^*_\ell}{\sigma}\Big)\bigg),$$
for all $r \in \{D_{m^*}+1,\cdots,q\}~:$ 
\begin{align*}
    \overline{f}_r(K,\beta^*,\sigma^2) = 1-\max\Bigg(&\underset{\ell \in \{1,\cdots,r-D_{m^*}\}}{\max}\Big(F_{\chi^2(\ell)}(\ell K)\Big), \\
    &\underset{\ell \in \{r-D_{m^*}+1,\cdots,r\}}{\max}\bigg(F_{\chi^2(\ell)}\Big(\frac{\ell K}{2} - \underset{k=r-\ell+1}{\overset{D_{m^*}}{\sum}}\frac{\beta^{*2}_k}{\sigma^2}\Big)\bigg)\Bigg),
\end{align*}
and all other terms are the same as those defined in Theorem~\ref{theorem}.
\label{corollary}
\end{corollary}

\vspace{0.3 cm}
A proof of Corollary~\ref{corollary} can be found in Subsection~\ref{proof_orthogonal}. \\
\begin{figure}[h!]
\begin{subfigure}{0.49\linewidth}
    \centering
    \includegraphics[width=1\linewidth]{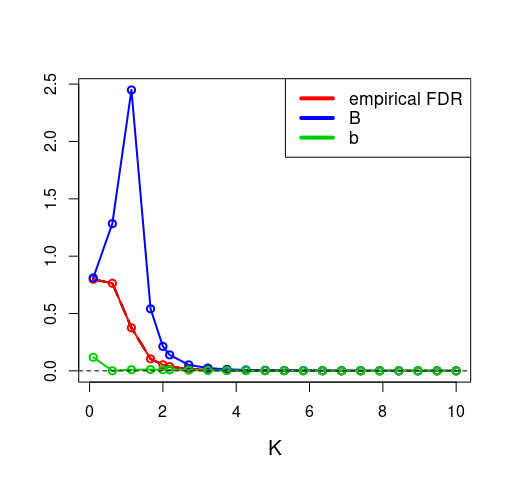}
\end{subfigure}
\begin{subfigure}{0.49\linewidth}
    \centering
    \includegraphics[width=1\linewidth]{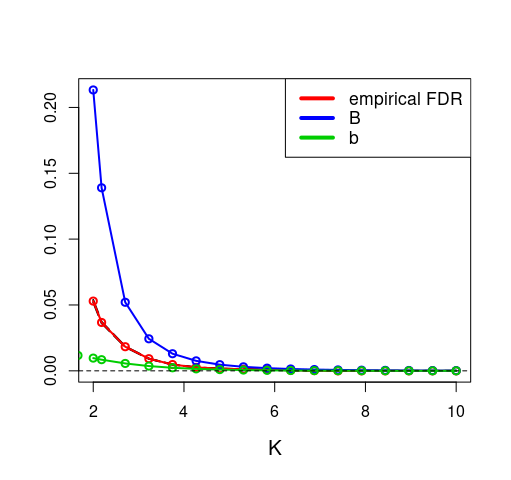}
\end{subfigure}
\caption{Left~: curves of the empirical values of $\text{FDR}(\Hat{m}(K))$ (red) and of terms $b(K,\beta^*,\sigma^2)$ (green) and $B(K,\beta^*,\sigma^2)$ (blue) for the orthogonal design matrix $X$ for the \textit{toy data set} described in Section~\ref{simulation_protocol}. Right~: curves are plotted only for $K\geq 2$.}
\label{upper_lower_bound_FDR_10}
\end{figure}

In Figure~\ref{upper_lower_bound_FDR_10}, we plot the empirical estimation of the $\text{FDR}(\Hat{m}(K))$ with the functions $b(K,\beta^*,\sigma^2)$ and $B(K,\beta^*,\sigma^2)$ on a grid of positive $K$ (left) and for $K \geq 2$ (right). Graphs are obtained from the \textit{toy data set} described in Section~\ref{simulation_protocol} where $X$ is orthogonal. The left figure is devoted to illustrate Corollary~\ref{corollary}. The FDR values are well smaller than the upper bound values and larger than the lower bound ones. From the right figure and in accordance with Corollary~\ref{asympt_K}, the empirical values of $\text{FDR}(\Hat{m}(K))$ tend to $0$ when $K$ increases and the convergence rate seems to be exponential. Moreover, the curves of $b(K,\beta^*,\sigma^2)$ and $B(K,\beta^*,\sigma^2)$ frame the empirical FDR and the difference between the three functions becomes quickly negligible for $K$ larger than $2$.

\section{Trade-off between the PR and the FDR controls.}
\label{trade_off_fixed_collection}
While bounds $b(K,\beta^*,\sigma^2)$ and $B(K,\beta^*,\sigma^2)$ are easily understandable and fully implementable, they depend on $\beta^*$, $\sigma^2$ and $D_m^*$. These quantities are unknown in practice. For a practical use, we propose to replace the theoretical bounds on the FDR as well as the theoretical expression of the PR with observable quantities (Subsection~\ref{data_driven_terms}). Then, we propose an algorithm to calibrate the hyperparameter $K$ from the data set such that both PR and FDR are controlled (Subsection~\ref{The_heuristic}). 

As variables are not usually naturally ranked, we explore the robustness of our algorithm under perturbations of the correct variable ordering and present approaches for obtaining a variable ordering in a data-driven manner. Results are provided in Subsection~\ref{non_ordered}. Lastly, our algorithm is compared with some existing variable selection procedures in Subsection~\ref{Comparaison_methods}, in terms of both PR and FDR.

\subsection{Estimation of the unknown quantities appearing in our bounds.}
\label{data_driven_terms}

For a practical use, we propose a new version of the predictive risk metric to evaluate the prediction performances of the selected estimator $\Hat{\beta}_{\Hat{m}(K)}$. The main advantage is that our approach does not require splitting the dataset in two sets (training and validation sets). The key is to compare $\Hat{\beta}_{\Hat{m}(K)}$ with $\Hat{\beta}_{\Hat{m}(2)}$ which is the benchmark (Section \ref{PR_estimated}). As for the theoretical bounds of the FDR, the unknown parameters $\beta^*, \sigma^2$ and $D_{m^*}$ have to be estimated. Using the simulation study, $\beta^*$ is estimated by $\Hat{\beta}_{\Hat{m}(4)}$, $D_{m^*}$ is estimated by $D_{\Hat{m}(4)}$ and $\sigma^2$ is estimated by using the slope heuristic method \cite{birge2007minimal} (Section \ref{parameter_estimations}).


\subsubsection{Estimation of the PR.}
\label{PR_estimated}
Commonly, the predictive risk is evaluated with the mean squared error on a validation set independent from the training set used to estimate the parameters (see Formula~(\ref{PR_def}) for the definition). However, it requires separating the dataset in two parts which increases the estimation errors. Here, we propose to use the entire dataset to both apply the model selection procedure and evaluate the predictive risk. Intuitively, the response vector $Y$ is replaced with $X\Hat{\beta}_{\Hat{m}(2)}$ which provides good prediction performances \cite{birge2007minimal}. Moreover, by re-expressing the PR, it is straightforward to show that for all $K>0$ and $K^{'}>0$~: 

\vspace{0.1 cm}
\begin{align}
\mathbb{E}[||Y - X\Hat{\beta}_{\Hat{m}(K)}||_2^2] &- \mathbb{E}[||Y - X\Hat{\beta}_{\Hat{m}(K^{'})}||_2^2] \notag \\
& = \mathbb{E}[||X\Hat{\beta}_{\Hat{m}(2)} - X\Hat{\beta}_{\Hat{m}(K)}||_2^2] - \mathbb{E}[||X\Hat{\beta}_{\Hat{m}(2)} - X\Hat{\beta}_{\Hat{m}(K^{'})}||_2^2] \notag \\
& \hspace{1 cm} - 2 \mathbb{E}[\langle X\Hat{\beta}_{\Hat{m}(2)} - Y, X\Hat{\beta}_{\Hat{m}(K^{'})}-X\Hat{\beta}_{\Hat{m}(K)} \rangle].
\label{new_risk}
\end{align}
According to \cite{birge2001gaussian}, the constant $2$ provides the optimal asymptotic control of~(\ref{theorical_rik}), so $||Y-X\Hat{\beta}_{\Hat{m}(2)}||_2$ is close to $0$. Moreover, $X\Hat{\beta}_{\Hat{m}(2)}$ is close to $\Pi_{\text{Im}(X)}(Y)$, so $X\Hat{\beta}_{\Hat{m}(2)} - Y$ almost belongs to the subspace $\text{Im}(X)^{\perp}$. In addition, $X\Hat{\beta}_{\Hat{m}(K)}$ and $X\Hat{\beta}_{\Hat{m}(K^{'})}$ belongs to $\text{Im}(X)$, so the last term in~(\ref{new_risk})
is close to $0$ and is negligible compared to the two others.
So, for all $K>0$ and $K^{'}>0$, $\mathbb{E}[||Y - X\Hat{\beta}_{\Hat{m}(K)}||_2^2] - \mathbb{E}[||Y - X\Hat{\beta}_{\Hat{m}(K^{'})}||_2^2]$ equals $\mathbb{E}[||X\Hat{\beta}_{\Hat{m}(2)} - X\Hat{\beta}_{\Hat{m}(K)}||_2^2] - \mathbb{E}[||X\Hat{\beta}_{\Hat{m}(2)} - X\Hat{\beta}_{\Hat{m}(K^{'})}||_2^2]$ up to an additive negligible term.
Hence, the constant $K$ minimizing $\mathbb{E}[||X\Hat{\beta}_{\Hat{m}(2)} - X\Hat{\beta}_{\Hat{m}(K)}||_2^2]$ and the one minimizing $\mathbb{E}[||Y - X\Hat{\beta}_{\Hat{m}(K)}||_2^2]$ are almost equal. Therefore, to evaluate the prediction performances of $\Hat{m}(K)$, we propose to compare prediction performances of the estimates $X\Hat{\beta}_{\Hat{m}(K)}$ and $X\Hat{\beta}_{\Hat{m}(2)}$. We introduce the following term that we call \textit{estimated difference in predictions}~: 
\begin{equation}
    \widehat{\text{diff-PR}}(\Hat{m}(K)) = \frac{1}{n} \underset{i=1}{\overset{n}{\sum}} \Big( \big(X\Hat{\beta}_{\Hat{m}(2)}\big)_i - \big(X\Hat{\beta}_{\Hat{m}(K)}\big)_i\Big)^2.
    \label{diff_PR}
\end{equation}
For the rest of the paper, the empirical version of~(\ref{diff_PR}) is calculated averaging over $100$ independent data sets and is denoted $\text{diff-PR}$. 
If this difference is significantly smaller than the noise level $\sigma^2$, the model $\Hat{m}(K)$ has performances similar to those satisfied by $\Hat{m}(2)$.

\subsubsection{Estimation of the FDR.}\label{parameter_estimations}
The functions $b(\cdot,\beta^*,\sigma^2)$ and $B(\cdot,\beta^*,\sigma^2)$ are explicit and easily implementable but depend on $\beta^*$, $\sigma^2$ and $D_m^*$, which are unknown.  \\
We propose~: 
\begin{enumerate}
    \item to apply the slope heuristic method \cite{birge2007minimal} to get an estimator $\Hat{\sigma}^2$ of $\sigma^2$, 
    \item to replace $\beta^*$ by the estimator $\Hat{\beta}_{\Hat{m}(4)}$,
    \item to replace $D_{m^*}$ by the number of non zero in $\Hat{\beta}_{\Hat{m}(4)}$.
\end{enumerate}
Choices of these estimators are crucial since they are proposed as inputs to the algorithm. Justifications of the choices are provided in the Supplementary file \cite{lacroix:hal-04625023} in which an extensive simulation study is presented.

\subsection{A data-dependent calibration of $K$ in the model selection procedure.}
\label{The_heuristic}
We propose a completely data-driven calibration of the hyperparameter $K$ up to $\alpha$ and $\gamma$. 
These parameters $\alpha$ and $\gamma$ are set by the user given values of the FDR and PR that are still deemed acceptable.
The algorithm depends on the  functions $K \longrightarrow B(\cdot,\Hat{\beta}_{\Hat{m}(4)},\Hat{\sigma}^2)$ and $K \longrightarrow \widehat{\text{diff-PR}}(\Hat{m}(K))$
to obtain a lower bound on both PR and FDR. \\
We propose the following algorithm~: 

\begin{algorithm}[H]
	\SetAlgoLined 
	
	\vspace{0.2 cm}
	\begin{enumerate}
        \item Choose $\alpha$ the threshold for the FDR control and $\gamma$ the threshold for the \textit{estimated risk} estimated difference in predictions~(\ref{diff_PR}).
        \item Compute $I_1 = \Big\{ K \geq 2, \ B(K,\Hat{\beta}_{\Hat{m}(4)},\Hat{\sigma}^2) \ \in \ ]0,\alpha[\Big\}$.
        \item 
        	Compute $I_2 = \Big\{ K \geq 2, \  \widehat{\text{diff-PR}}(\Hat{m}(K)) < \gamma \times \Hat{\sigma}^2 \Big\} $.
        \item If $I_1 \cap I_2 \neq \emptyset$, return $\min \Big\{K, \ K \in I_1 \cap I_2 \Big\}$ \;
        Otherwise, return $\min \Big\{K, K \ \in I_1 \Big\}$ or take a larger value of either $\alpha$ or $\gamma$.
    \end{enumerate}
	\caption{Algorithm to calibrate $K$}
	\label{algo_FDR}
\end{algorithm}  


\begin{figure}[h!]
\begin{subfigure}{0.49\linewidth}
    \centering
   \includegraphics[width=1\linewidth]{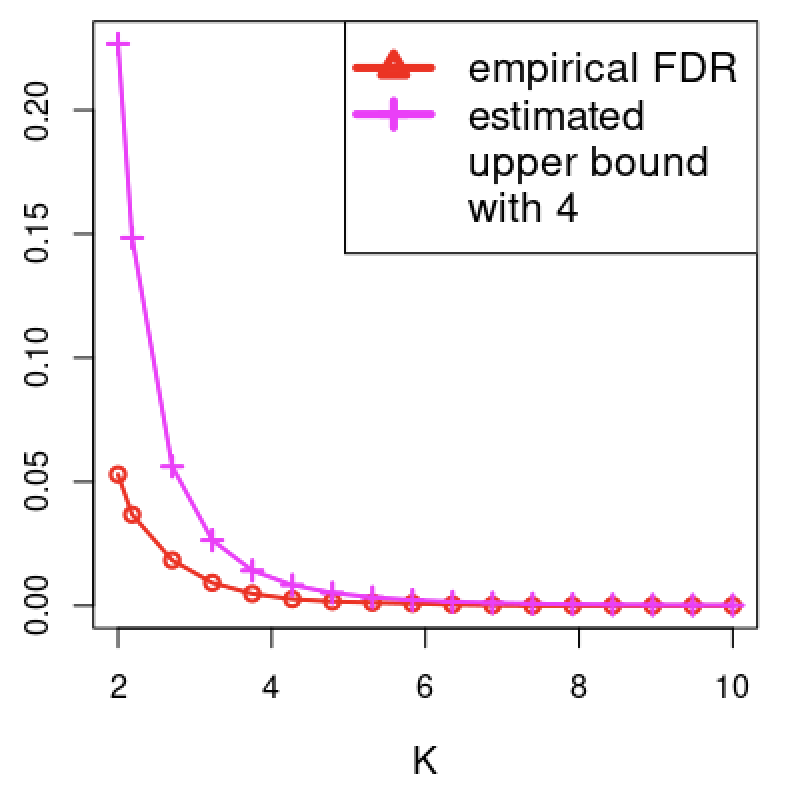}
    \caption*{FDR}
\end{subfigure}
 \begin{subfigure}{0.49\linewidth}
    \centering
    \includegraphics[width=1\linewidth]{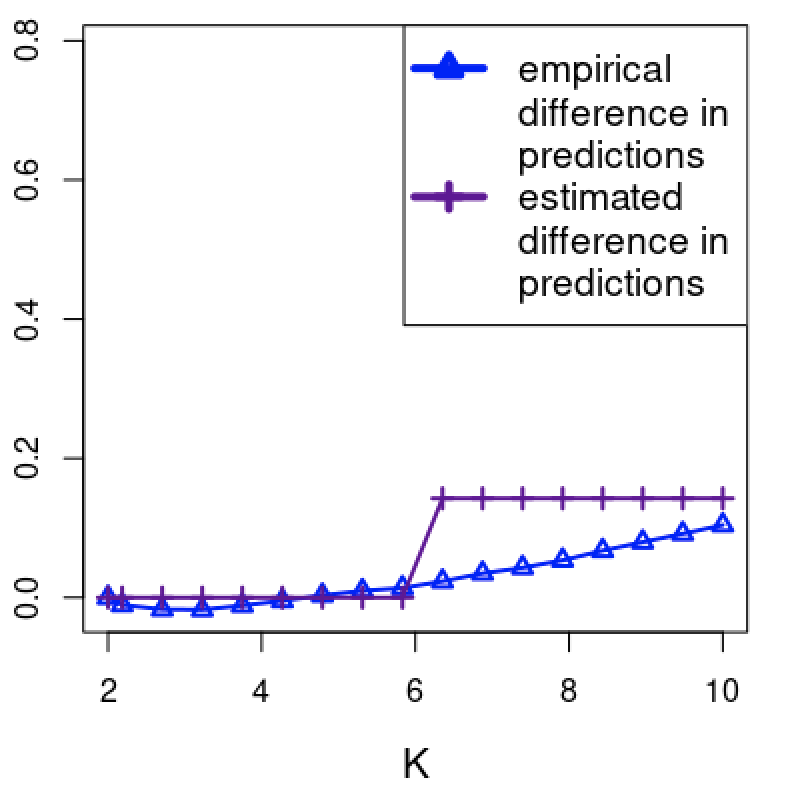}
 \caption*{diff-PR}
\end{subfigure} 
\begin{subfigure}{0.29\linewidth}
    \centering
   \includegraphics[width=1\linewidth]{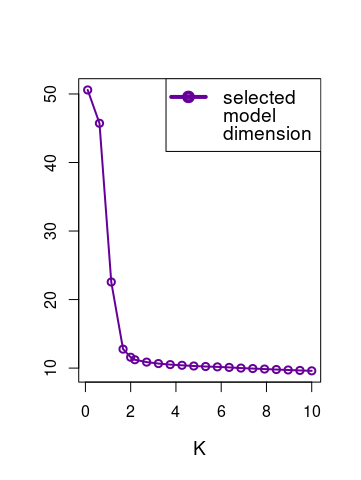}
    \caption*{selected model per dim}
\end{subfigure}
\hfill
\begin{subfigure}{0.6\linewidth}
    \centering
    \includegraphics[width=1\linewidth]{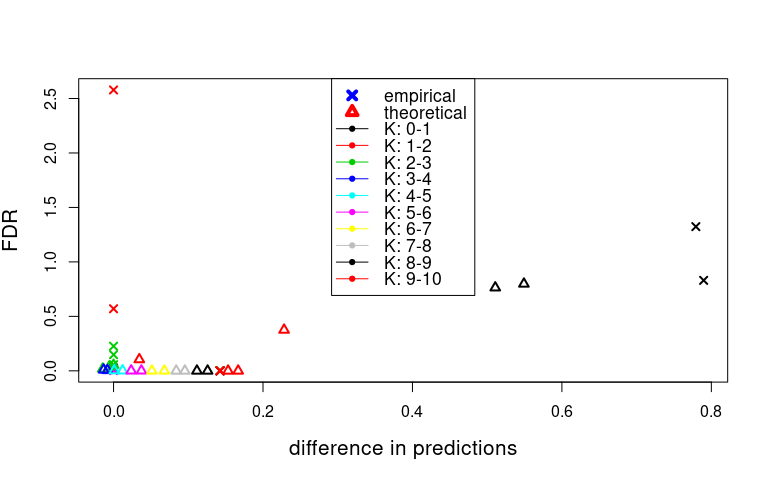}
    \caption*{$\text{FDR}$ and $\widehat{\text{FDR}}$ as function of diff-PR and  $\widehat{\text{diff-PR}}$}
\end{subfigure}
\caption{Top~: Curves of the empirical functions $\text{FDR}\big(\Hat{m}(K)\big)$ (red) and diff-PR$\big(\Hat{m}(K)\big)$ (blue), of the $B(K,\Hat{\beta}_{\Hat{m}(4)},\Hat{\sigma}^2)$ functions (pink) and of  $\widehat{\text{diff-PR}}\big(\Hat{m}(K)\big)$ (violet) for $K \geq 2$ for the \textit{toy data set}. Bottom~: Curves of the $D_{\Hat{m}(K)}$ as function of $K$ averaged over the $1000$ data sets (left) and values of the empirical $\text{FDR}\big(\Hat{m}(K)\big)$ and $\widehat{\text{FDR}}\big(\Hat{m}(K)\big)$ as functions of diff-PR$\big(\Hat{m}(K)\big)$ and  $\widehat{\text{diff-PR}}\big(\Hat{m}(K)\big)$ (right) for all $K>0$ and for the \textit{toy data set}.}
\label{estimated_PR_FDR}
\end{figure}


Curves of Figure~\ref{estimated_PR_FDR} are generated from the \textit{toy data set} and from the simulation protocol described in Section~\ref{simulation_protocol}. 
Parameters $\alpha$ and $\gamma$ are free and are defined by the user for maximum acceptable values for FDR and PR. In
this example, we choose $\alpha=0.05$ and $\gamma = 0.1$.
The graph at the bottom right shows that there exists some constants $K$ for which a trade-off between both theoretical FDR and PR and both empirical FDR and PR can be achieved. Here, the range of $K$ values is given by the interval $[2,6]$. These values correspond to a selected model dimension close to $D_{m^*}$ (Figure~\ref{estimated_PR_FDR} at  the bottom left). By applying our algorithm on this example, we get $I_1 = [3.3,10]$ and $I_2 = [2,5.8]$ and so, our proposed algorithm returns $K=3.3$.
The evaluation of the prediction performances provided by the selected model $\Hat{m}(3.3)$ is equal to $1.14$ 
and we get $B(3.3,\Hat{\beta}_{\Hat{m}(4)},\Hat{\sigma}^2)=0.03$. The constant $K=3.3$ corresponds to a low value of both empirical predictive risk and FDR functions. Indeed, the empirical predictive risk of $\Hat{m}(3.3)$ is equal to $1.24$ and the empirical FDR of $\Hat{m}(3.3)$ is equal to $0.01$. To compare with the usual choice $K=2$, the empirical predictive risk of $\Hat{m}(2)$ is equal to $1.25$ and the empirical FDR of $\Hat{m}(2)$ is equal to $0.05$. Hence, our proposed algorithm allows to maintain the prediction performances from $\Hat{m}(2)$, reinforce the control of the FDR criterion and so gain a convenient trade-off between PR and FDR.

In the supplementary file \cite{lacroix:hal-04625023},
the algorithm~\ref{algo_FDR} is applied to several data sets generated from various sets of parameters and described in Table~\ref{Table_scenario}. Each time, the hyperparameter $K$ is strictly larger than the commonly used constant $2$ and provides a low value of FDR while maintaining the prediction performances given by $\Hat{m}(2)$. 

\subsection{Towards non-ordered variable selection}
\label{non_ordered}
For most applications, no canonical order of variables is available and our algorithm cannot be applied directly. We propose to generate candidate orders from random procedures to use our method when an ordering of the variables is not given a priori. \\
More precisely, we first study the robustness to variable ordering of our method (Subsection~\ref{robustness}) and provide some procedures to construct variable orders in practice (Subsection~\ref{random_collection}). Our algorithm~\ref{algo_FDR} is then applied from the generated rankings in Subsection~\ref{Comparaison_methods}.

\subsubsection{Robustness to variable ordering}
\label{robustness}
We propose numerical experiments where the assumption of ordered variables is not fulfilled. The goal is to test the robustness to variable ordering of our algorithm by measuring how this impacts its performances. 
We consider the \textit{toy data set} where the size of the true model is $D_{m^{*}} = 10$ and we consider three collections which are the results of a random permutation of the nested model collection (\ref{nested_collection}) on respectively the first ten, the first twelve and the first fifteen variables.
Hence, active variables remain first in the first collection; perturbations may introduce non-active variables among the first ten variables in the second collection, while in the third collection, some active variables can be pushed far into the collection.

\begin{figure}[h!]
\begin{subfigure}{0.49\linewidth}
    \centering
    \includegraphics[width=1\linewidth]{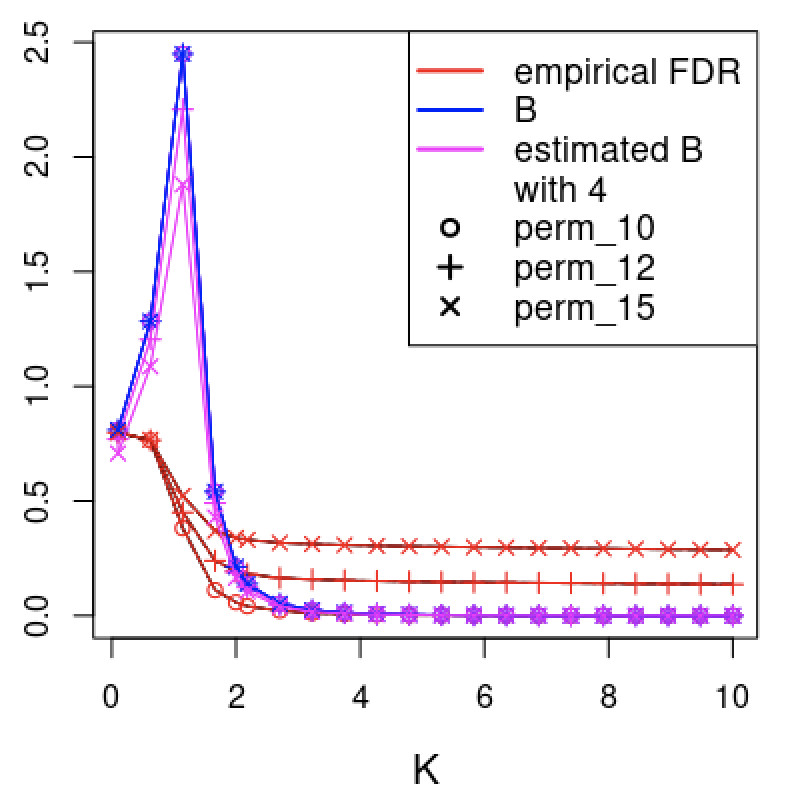}
 \caption*{FDR}
\end{subfigure}
 \begin{subfigure}{0.49\linewidth}
    \centering
    \includegraphics[width=1\linewidth]{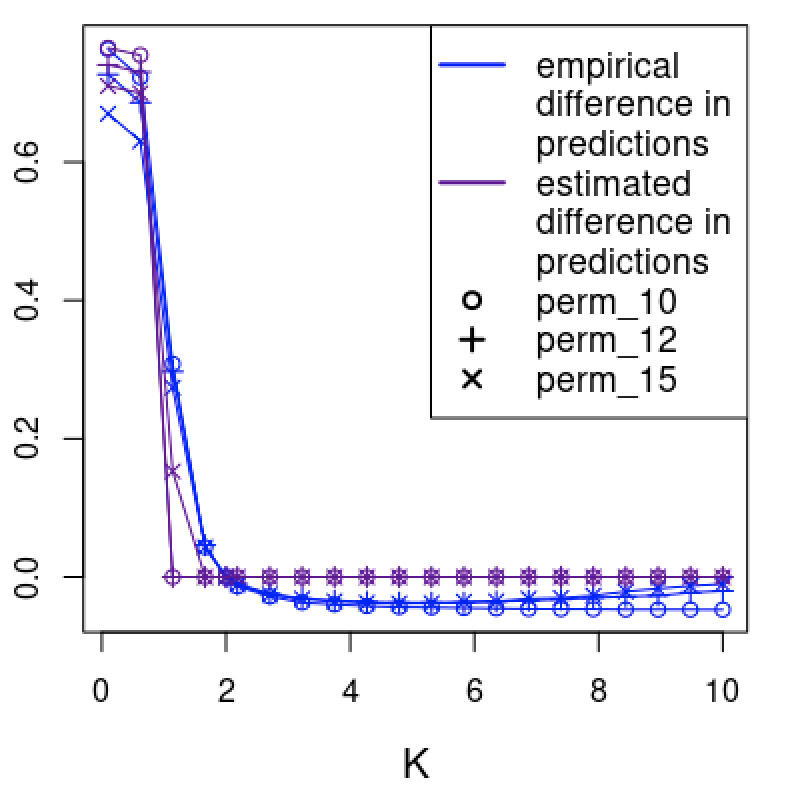}
 \caption*{diff-PR}
\end{subfigure}   
\caption{Curves of the empirical functions $\text{FDR}\big(\Hat{m}(K)\big)$ (red) and diff-PR$\big(\Hat{m}(K)\big)$ (blue), of the $B(K,\beta^{*},\sigma^2)$ functions (blue), the $B(K,\Hat{\beta}_{\Hat{m}(4)},\Hat{\sigma}^2)$ functions (pink) and $\widehat{\text{diff-PR}}\big(\Hat{m}(K)\big)$ (violet) for the \textit{toy data set} and for the three perturbed collections.}    
\label{robustness_order}
\end{figure}
To test the robustness to variable ordering of our algorithm, Figure~\ref{robustness_order} shows how the empirical values of FDR behaves in relation to its estimated upper bound as well as the empirical and estimated differences in predictions for the three perturbed collections.
We observe
that when the permutation concerns only the active variables (on the nested model collection (\ref{nested_collection})), values of the empirical values of FDR are smaller than the values of $B(K,\beta^*,\sigma^2)$ and $B(K,\Hat{\beta}_{\Hat{m}(4)},\Hat{\sigma}^2)$ which are close. For prediction, the $\text{diff-PR}$ function has the same behavior than for the nested model collection (\ref{nested_collection}).\\
When the permutation concerns the first twelve and the first fifteen variables, the empirical values of FDR is higher than $B(K,\beta^*,\sigma^2)$ and $B(K,\Hat{\beta}_{\Hat{m}(4)},\Hat{\sigma}^2)$ as soon as $K \geq 2$ and with an increasing
deviation when the error on estimated variable order increases.
Moreover, we observe that the rate of the empirical values of FDR decay is much slower and values are high whatever the value of $K$~:
above $0.13$ and $0.28$ for the second and third perturbed model collection,
respectively. For prediction, the $\text{diff-PR}$ function 
is stable for $K \geq 2$.\\
Hence, permutations among only the active variables have no effect on FDR and PR. However, as soon as a non-active variable is ranked before an active variable, the theoretical guarantees of Theorem~\ref{theorem} no longer hold and empirical values of FDR can be high whatever the value of $K$. 
To tackle this problem, one solution consists of combining our algorithm with a method discriminating active and non-active variables. We consider this direction for the rest of this section.

\subsubsection{Random variable order}
\label{random_collection}
We consider four strategies to estimate variable orders. The \textit{Bolasso} procedure \cite{bach2008bolasso} consists of solving the Lasso equation~(\ref{lasso_eq}), through the LARS algorithm \cite{efron2004least}, on several resamples and for different values of $\lambda$. Variables are ranked according to their occurrence frequency in the models averaged over the $\lambda$'s and the resamples. The \textit{random forests} \cite{breiman2001random} are aggregation of several binary decision trees. The tree predictors are generated on bootstrap resamples and on a subset of variables randomly chosen. Here, we combine the random forest with the \textit{recursive feature elimination (RFE) algorithm} \cite{guyon2002gene} whose efficiency has been proved especially for correlated variables \cite{gregorutti2017correlation}.
Variables are ordered according to their importance defined by the random forest. 
The Sorted $\ell_1$ penalized estimator (\textit{SLOPE}) \cite{bogdan2013statistical} is obtained by solving the Lasso equation~(\ref{lasso_eq}) with $\lambda$ a $p$-vector calculated from a multiple testing procedure. Lastly, \textit{the knockoff method} \cite{barber2015controlling} consists of building a non-active copy $\Tilde{X}_j$ of each $X_j$ and solving the Lasso equation~(\ref{lasso_eq}) for several values of $\lambda$ on the augmented matrix composed on the $X_j$ and $\Tilde{X}_j$ variables. Variables are then sorted according to the values of 
\begin{equation*}
    W_j = \max\left(Z_j,\Tilde{Z}_j\right) \times \text{sign}\left(Z_j-\Tilde{Z}_j\right),
\end{equation*}
for all $j \in \{1,\cdots,p\}$, where $Z_j$ and $\Tilde{Z}_j$ correspond to the largest $\lambda$ for which $X_j$ and $\Tilde{X}_j$ are respectively selected. 
For each strategy, a random model collection, on which model selection can be applied is defined from the estimated variable order.
Bolasso and random forest both provide a variable order from a prediction point of view, whereas SLOPE and the knockoff method provide a variable order by considering both PR and FDR controls.

To quantify the ability to discriminate between active and non-active variables, we calculate the proportion of active variables in models of size $5$, $10$, $15$ and $20$ of each random collection. Results are presented in Table~\ref{proportion_active} where each value is the average over $100$ independent iterations. With Bolasso, random forests and SLOPE, $50$ resamples for the construction of random collections are considered.

\begin{table}[ht]
\centering
\begin{tabular}{|m{2.5cm}||m{1.8cm}|m{1.8cm}|m{1.8cm}|m{1.8cm}|}
  \hline
  & Bolasso & SLOPE & random forests & the knockoff \newline method \\ 
  \hline
  $D_m$ = 5 & 0.99 & 0.99 & 0.98 & 1.00 \\ 
  $D_m$ = 10 & 0.83 & 0.83 & 0.82 & 0.85 \\ 
  $D_m$ = 15 & 0.92 & 0.92 & 0.90 & 0.90 \\ 
  $D_m$ = 20 & 0.95 & 0.95 & 0.93 & 0.92 \\ 
   \hline
\end{tabular}
\caption{Proportion of active variables in models of size $5$, $10$, $15$ and $20$ for random collections built with Bolasso, SLOPE, random forest and the knockoff method. Each value is the average over $100$ independent iterations.}
\label{proportion_active}
\end{table}

The collection built by the knockoff method is the collection containing the fewest non-active variables in the models of size $5$ and $10$.   
For models of size $15$ and $20$, Bolasso and SLOPE are slightly better and the proportion of active variables is always larger than $0.9$. We observe with the model of size $20$ that there are active variables far away in the collections, which is undesirable. Since Bolasso is slightly better than SLOPE on scenarios described in Table~\ref{Table_scenario} (see the supplementary file \cite{lacroix:hal-04625023}), we consider the random collections built with the knockoff method and Bolasso in the following.

\subsection{Comparison with other variable selection methods}
\label{Comparaison_methods}
Performances of Algorithm~\ref{algo_FDR} are compared with three variable selection procedures. 
The \textit{LinSelect} penalty \cite{giraud2012high}
is a model selection criterion introduced in a non-asymptotic setting to take into account of the randomness of the model collection. 
The penalty function provides a sharp oracle inequality.
The \textit{$V$-fold cross-validation} \cite{allen1974relationship,stone1974cross,shao1993linear} is the most popular, adaptive and simple variable selection method. The final selected model is the one with the best prediction performance accuracy over the data sets obtained by splitting the initial data set into a training set and a validation set.
The last method is the knockoff method where the final variable subset is composed by $X_j$ such that $W_j \geq T$ where $T$ is defined to satisfy a given control of FDR. LinSelect and $V$-fold cross-validation aim at providing a control of PR while the knockoff method aims at providing a control of FDR.

We consider the $50$-fold cross-validation and evaluate PR and FDR of our algorithm and of the three variable selection procedures on the nested model collection (\ref{nested_collection}) where the active variables are properly ranked before the non-active variables
and on random collections built with Bolasso and with the knockoff method. \\
\begin{table}[ht]
\centering
\begin{tabular}{|m{5cm}||m{1.5cm}|m{1.5cm}|m{1.5cm}|}
  \hline
 & $D_{\Hat{m}}$ & PR($\Hat{m}$) & FDR($\Hat{m}$) \\ 
  \hline
   \textbf{nested model collection} & & &   \\ 
   \hline
LinSelect & 8.86 & 1.35 & 0.01 \\ 
  $50$-fold CV & 26.42 & 2.29 & 0.45 \\ 
  Knockoff &  &  &  \\ 
  Our algorithm & 9.37 & 1.25 & 0.00 \\ 
   \hline
   \textbf{Bolasso collection} & & &   \\ 
   \hline
    LinSelect & 10.17 & 1.93 & 0.07 \\ 
  $50$-fold CV & 22.10 & 2.77 & 0.37 \\ 
  Knockoff &  &  &  \\ 
  Our algorithm & 13.96 & 1.59 & 0.25 \\ 
       \hline
        \textbf{the knockoffs collection} & & &   \\ 
   \hline
    LinSelect & 8.86 & 1.81 & 0.03 \\ 
  $50$-fold CV & 20.85 & 2.45 & 0.35 \\ 
  Knockoff & 0.00 & 14.10 & 0.00 \\ 
  Our algorithm & 13.33 & 1.65 & 0.18 \\ 
    \hline 
\end{tabular}
\caption{Results of the dimension, PR and FDR of the selected models obtained by LinSelect, the $50$-fold CV, the knockoff method and our algorithm, applied on the nested model collection~(\ref{nested_collection}) and on random collections built with Bolasso and the knockoff method. Each value is the average over $100$ independent iterations. PR and FDR of each selected model are the empirical quantities. Input parameters of our algorithm are fixed to $\gamma = 0.1$ and $\alpha=0.05$. Note that the knockoff variable selection method is adapted for only the knockoff random model collection. 
}
\label{performances_selection}
\end{table}

Table~\ref{performances_selection} shows the performances of the four variable selection procedures. As the knockoff method is a procedure for both collection generation and variable selection, the knockoff collection is only used with the knockoff variable selection method. On the nested model collection~(\ref{nested_collection}),
our algorithm provides the smallest values in both FDR and PR and the average of the selected model sizes is the closest to the true model size. LinSelect behaves in a similar way while the $50$-fold cross validation selects a model located in the over-fitting area providing a high value of both PR and FDR. 
On random model collections, performances are deteriorated for all methods, as expected,
and are slightly better on model collections built with the knockoff method than with Bolasso. Our algorithm provides the smallest PR but LinSelect provides the smallest FDR. While LinSelect is designed to control the PR theoretically, we remark that it is apparently also a relevant candidate to control FDR and to achieve a trade-off between both PR and FDR. The $50$-fold cross validation method provides poor results while the knockoff method 
selects the empty set of variable. The size of the model selected by our algorithm is larger when the collections are random 
and provides high values of FDR. 
These results show that a meticulous choice of 
$\gamma$ and $\alpha$ is important to improve our algorithm performances. 

The robustness of our algorithm to variable order, the construction of random model collections and performances of the four variable selection procedures are studied on several data sets generated from various sets of parameters and described in Table~\ref{Table_scenario}. Results are presented in the supplementary file \cite{lacroix:hal-04625023},
and the conclusions remain the same.

\section{Conclusions.}
\label{conclusion_FDR}
The variable selection procedure in a high-dimensional Gaussian linear regression with sparsity assumption is commonly used to identify a set of variables with prediction performances or with as few non-active variables as possible. 
	For prediction performances, the PR is usually controlled via a penalized least-squares minimization; to avoid the selection of non-active variables, the FDR is usually controlled via a multiple testing approach.
Controlling the PR tends to select too many variables, including non-active ones, whereas controlling the FDR tends to select too few variables, leaving out some active ones.

This work shows that a convenient trade-off between PR and FDR can be achieved in ordered variable selection. The originality of this paper is to obtain this trade-off through a proper calibration of the hyperparameter $K$ in the penalty of the model selection~(\ref{our_penalty_FDR}).
Firstly, theoretical results lead to non-asymptotic lower and upper bounds on the $\text{FDR}\big(\Hat{m}(K)\big)$ function when $\sigma^2$ is known.
Asymptotic behaviors suggest that bounds are optimal. 
Secondly, the proposed methodology provides an algorithm to calibrate the hyperparameter $K$ in the penalty function when $\sigma^2$ is unknown. This algorithm is based on completely data-driven terms~: the estimated difference in predictions and the estimated upper bound on the FDR where the choices of estimators $\Hat{\sigma}^2$ and $\Hat{\beta}_{\Hat{m}(4)}$ are derived from an extensive simulation study. 
The hyperparameter $K$ is calibrated from the dataset to ensure $\widehat{\text{diff-PR}}(\Hat{m}(K)) < \gamma \times \Hat{\sigma}^2$ under the constraint $B(K,\Hat{\beta}_{\Hat{m}(4)},\Hat{\sigma}^2) < \alpha$. 
Our algorithm is validated on an extensive simulation study and allows to obtain a selected model ensuring a small value of both theoretical PR and FDR. The calibrated hyperparameter $K$ is strictly larger than the commonly used constant $K=2$. Moreover, PR and FDR values of the selected model with our algorithm are the smallest values compared with the existing variable selection procedures considered in the paper.
Lastly, we propose a preliminary response to construct a random model collection to extend our work to non-ordered variable selection. The performances of our algorithm deteriorate as soon as a non-active variable is ranked before an active one, but combined with procedures with high ability to discriminate between active and non-active variables, our algorithm is competitive with some existing variable selection procedures. 

If the selected model is the largest one of the nested model collection (with dimension equals $q$), the associated lower and upper bounds values equal $0$. In this case, a distinction between $D_{m^*} = q$ and $D_{m^*} < q$ is not possible without additional arguments. This is a limitation of our work.

The main perspective of our work is to generalize our theoretical results to non-ordering variable selection. The ordered variable assumption is the key ingredient of our proofs and appears in the second line of the proof where the ratio is fixed, allowing randomness only on $\Hat{m}$ that we control thanks to the ordered model selection theory. Hence, relaxing this assumption requires new technical arguments and this is a real challenge for future work. Moreover, for non-ordered variable selection, the penalty function~(\ref{our_penalty_FDR}) 
has to include a logarithmic term 
to take into account that all possible models should be explored but this is computationally infeasible given the combinatorial nature of the problem. 
In this case, two hyperparameters have to be calibrated. Another way to generalize our work to non-ordered variable selection is to 
detect the value of $K$ from which the theoretical FDR is larger than the theoretical upper bound on the FDR and quantify the gap between the theoretical upper bound and the FDR. 

One way to improve the performances of our algorithm can be a meticulous choice of the algorithm input parameters $\alpha$ and $\gamma$, which are arbitrarily fixed in our work. \\
Achieving a trade-off between FDR and PR is not trivial and investigating alternatives in this direction 
can be considered in future work. In particular, in this work, the model collection is constructed from a prediction point of view method (minimization of the least squares values). It could be judicious to introduce the FDR metric from the step of model collection generation. \\
A possible opening is to study the potential characteristics of the hyperparameter $K$ provided by our algorithm in a theoretical point of view (dependence in $\beta^*$ and $\sigma^2$). 
Another possible extension is to study the false negative rate (FNR) function in the model selection procedure, similarly and in addition to the FDR one. This can provide a more powerful method, 
similarly to 
\cite{genovese2002operating, genovese2004stochastic}. \\
Finally, another generalization is to extend our theoretical results to unknown variance, random model collections or to non-fixed designs, which are more general frameworks adapted to some application points of view. These extensions are much more intricate. \\

\newpage
\section{Proofs of theoretical results.}
\label{proofs}
This section contains proofs of all the theoretical results of this paper.
\subsection{FDR expression in model selection.}
\label{proof_FDR_expression}
\textbf{\textit{Proof of Formula~\ref{FDR_with_crit}}.} \\
If $D_m^*=q$, then $\text{FP}(m) = 0$ for all $m \in \mathcal{M}$ and $\text{FDR}(m) = 0$ for all $m \in \mathcal{M}$. \\
Let us now suppose that $D_m^*<q$. 
The FDP expression within the model selection procedure is~: \\
\begin{align*}
    \forall K>0, \quad \quad \text{FDP}(\Hat{m}(K)) &= \frac{\text{FP}(\Hat{m}(K))}{\max(D_{\Hat{m}(K)},1)} \\
    &\underset{\textit{(*)}}{=} \quad \frac{D_{\Hat{m}(K)}-D_{m^{*}}}{D_{\Hat{m}(K)}} \ind{\{D_{\Hat{m}(K)}>D_{m^{*}}\}} \\
    &= \underset{r = 1}{\overset{q}{\sum}} \frac{r-D_{m^{*}}}{r}  \ind{\{r>D_{m^{*}}\}} \ind{\{D_{\Hat{m}(K)}=r\}} \\
    &\underset{\textit{(**)}}{=} \underset{r = D_{m^*}+1}{\overset{q}{\sum}} \frac{r-D_{m^*}}{r} \ind{\left\{\Hat{m}(K) = m_r\right\}} \\
    &\underset{\textit{(***)}}{=} \underset{r = D_{m^*}+1}{\overset{q}{\sum}} \frac{r-D_{m^*}}{r} \ind{\left\{\underset{\underset{\ell \neq r}{\ell = 0}}{\overset{q}{\cap}} \{ \text{crit}_K(m_r) < \text{crit}_K(m_\ell)\}\right\}}.
\end{align*}
\textit{(*)} and \textit{(**)} are due to the fact that models $\left(m\right)_{m \in \mathcal{M}}$ are nested and $m^* \in \mathcal{M}$. \textit{(***)} is obtained since the $\text{crit}_K$ function is injective on $\mathcal{M}$. Finally, by taking the expectation, we obtain the FDR expression~(\ref{FDR_with_crit}). 
\qed

\vspace{0.3 cm}
\textbf{\textit{Proof of Proposition~\ref{Proposition}}.} \\
Before proving Proposition~\ref{Proposition}, let us cite and prove two lemmas.
\begin{lemma}
For $r \in \{D_{m^*}+1, \cdots,q\}$ and for all $\ell \in \{0,\cdots,r-1\}$~:
$$||Y-X\Hat{\beta}_{m_r}||_2^2 - ||Y-X\Hat{\beta}_{m_\ell}||_2^2 = - \underset{k=\ell+1}{\overset{r}{\sum}} \langle Y, u_k \rangle^2.$$
\label{Lemma_1}
\end{lemma}

\vspace{0.3 cm}
\begin{lemma}
For $r \in \{D_{m^*}+1, \cdots,q\}$ and for all $\ell \in \{r+1,\cdots,q\}$~:
$$||Y-X\Hat{\beta}_{m_r}||_2^2 - ||Y-X\Hat{\beta}_{m_\ell}||_2^2 = \underset{k=r+1}{\overset{\ell}{\sum}} \langle Y, u_k \rangle ^2.$$
\label{Lemma_2}
\end{lemma}

\vspace{0.3 cm}
\textbf{\textit{Proof of Lemma~\ref{Lemma_1}}.}\\
For $r \in \{D_{m^*}+1,\cdots,q\}$ and $\ell \in \{0,\cdots,r-1\}$~:
\begin{align*}
    &||Y-X\Hat{\beta}_{m_r}||_2^2 - ||Y-X\Hat{\beta}_{m_\ell}||_2^2 = ||X\Hat{\beta}_{m_r}||_2^2 - ||X\Hat{\beta}_{m_\ell}||_2^2 + 2 \langle Y,X\Hat{\beta}_{m_\ell}-X\Hat{\beta}_{m_r}\rangle \\
    &= ||X\Hat{\beta}_{m_r}||_2^2 - ||X\Hat{\beta}_{m_\ell}||_2^2 + 2 \langle Y-X\Hat{\beta}_{m_r},X\Hat{\beta}_{m_\ell}\rangle \\
    & \hspace{3 cm} - 2\langle Y-X\Hat{\beta}_{m_r},X\Hat{\beta}_{m_r} \rangle + 2 \langle X\Hat{\beta}_{m_r},X\Hat{\beta}_{m_\ell} \rangle - 2 ||X\Hat{\beta}_{m_r}||_2^2 \\
    &= -||X\Hat{\beta}_{m_r}||_2^2 - ||X\Hat{\beta}_{m_\ell}||_2^2 + 2 \langle X\Hat{\beta}_{m_r},X\Hat{\beta}_{m_\ell} \rangle = - ||X\Hat{\beta}_{m_r}-X\Hat{\beta}_{m_\ell}||_2^2.
\end{align*}
The last line is due to the fact that $Y-X\Hat{\beta}_{m_r} \in (m_r)^{\perp} \subset (m_\ell)^{\perp}$ since $m_\ell \subset m_r$ and $X\Hat{\beta}_{m_r}$ is the projection of $Y$ onto $m_r$. \\
Then, 
\begin{align*}
    ||X\Hat{\beta}_{m_r}-X\Hat{\beta}_{m_\ell}||_2^2 &= ||\Pi_{m_r}(Y) - \Pi_{m_\ell}(Y)||_2^2 \\
    &= ||\Pi_{\text{Span}(X_1,\cdots,X_r)}(Y) - \Pi_{\text{Span}(X_1,\cdots,X_\ell)}(Y)||_2^2  \\
    &\underset{\textit{(*)}}{=} ||\Pi_{\text{Span}(u_1,\cdots,u_r)}(Y) - \Pi_{\text{Span}(u_1,\cdots,u_\ell)}(Y)||_2^2  \\
    &= ||\Pi_{\text{Span}(u_{\ell+1},\cdots,u_r)}(Y)||_2^2 \\
    & = || \underset{k=\ell+1}{\overset{r}{\sum}} \langle Y,u_k \rangle u_k ||_2^2  \\
    &\underset{\textit{(**)}}{=} \underset{k=\ell+1}{\overset{r}{\sum}} \langle Y, u_k \rangle ^2.
\end{align*}
\textit{(*)} come from the definition of $\left(u_1,\cdots,u_n\right)$ and \textit{(**)} is obtained by Parseval's identity. \\
\qed

\vspace{0.3 cm}
\textbf{\textit{Proof of Lemma~\ref{Lemma_2}}.}\\
For $r \in \{D_{m^*}+1,\cdots,q\}$ and $l \in \{r+1,\cdots,q\}$~:
\begin{align*}
    ||Y-X\Hat{\beta}_{m_r}||_2^2 - ||Y-X\Hat{\beta}_{m_\ell}||_2^2 &= ||X\Hat{\beta}_{m_r}||_2^2 - ||X\Hat{\beta}_{m_\ell}||_2^2 + 2 \langle Y,X\Hat{\beta}_{m_\ell}-X\Hat{\beta}_{m_r}\rangle \\
    &= ||X\Hat{\beta}_{m_r}||_2^2 - ||X\Hat{\beta}_{m_\ell}||_2^2 + 2 \langle Y-X\Hat{\beta}_{m_\ell},X\Hat{\beta}_{m_\ell}\rangle \\
    &\quad \quad - 2\langle Y-X\Hat{\beta}_{m_\ell},X\Hat{\beta}_{m_r} \rangle + 2 ||X\Hat{\beta}_{m_\ell}||_2^2 \\
    &\quad \quad - 2 \langle X\Hat{\beta}_{m_\ell},X\Hat{\beta}_{m_r} \rangle  \\
    &\underset{\textit{(*)}}{=} ||X\Hat{\beta}_{m_r}||_2^2 + ||X\Hat{\beta}_{m_\ell}||_2^2 - 2 \langle X\Hat{\beta}_{m_\ell},X\Hat{\beta}_{m_r} \rangle \\
    &=  ||X\Hat{\beta}_{m_\ell}-X\Hat{\beta}_{m_r}||_2^2.
\end{align*}
\textit{(*)} is due to the fact that $Y-X\Hat{\beta}_{m_\ell} \in (m_\ell)^{\perp} \subset (m_r)^{\perp}$ since $m_r \subset m_\ell$, and $X\Hat{\beta}_{m_\ell}$ is the projection of $Y$ onto $m_\ell$.\\
Then, 
\begin{align*}
    ||X\Hat{\beta}_{m_\ell}-X\Hat{\beta}_{m_r}||_2^2 &= ||\Pi_{m_\ell}(Y) - \Pi_{m_r}(Y)||_2^2 \\
    &= ||\Pi_{\text{Span}(X_1,\cdots,X_\ell)}(Y) - \Pi_{\text{Span}(X_1,\cdots,X_r)}(Y)||_2^2 \\
    &\underset{\textit{(*)}}{=} ||\Pi_{\text{Span}(u_1,\cdots,u_\ell)}(Y) - \Pi_{\text{Span}(u_1,\cdots,u_r)}(Y)||_2^2 \\
    &= ||\Pi_{\text{Span}(u_{r+1},\cdots,u_\ell)}(Y)||_2^2 \\
    &= || \underset{k=r+1}{\overset{\ell}{\sum}} \langle Y,u_k \rangle u_k ||_2^2   \\
    &\underset{\textit{(**)}}{=} \underset{k=r+1}{\overset{\ell}{\sum}} \langle Y, u_k \rangle ^2.
\end{align*}
\textit{(*)} come from the definition of $\left(u_1,\cdots,u_n\right)$ and \textit{(**)} is obtained by Parseval's identity. \\
\qed

\vspace{0.3 cm}
\textbf{\textit{Proof of Proposition~\ref{Proposition}}.}\\
Starting from~(\ref{FDR_with_crit}), we decompose the event $\left\{\underset{\underset{\ell \neq r}{\ell = 0}}{\overset{q}{\cap}} \{ \text{crit}_K(m_r) < \text{crit}_K(m_\ell)\}\right\}$ by the intersection of these two events \\
$\left\{\underset{\ell = 0}{\overset{r-1}{\cap}} \{ \text{crit}_K(m_r) < \text{crit}_K(m_\ell)\}\right\}$ and $\left\{\underset{\ell = r+1}{\overset{q}{\cap}} \{ \text{crit}_K(m_r) < \text{crit}_K(m_\ell)\}\right\}$. \\
By using the definition of the $\text{crit}_K$ function, we have for $r \in \{D_{m^*}+1,\cdots,q\}$ and $\ell \in \{0,\cdots,r-1,r+1,\cdots,q\}$~:
\begin{align*}
    \Big\{ \text{crit}_K(m_r) < \text{crit}_K(m_\ell) \Big\} &= \Big\{ ||Y-X\Hat{\beta}_{m_r}||_2^2 + K \sigma^2 r < ||Y-X\Hat{\beta}_{m_\ell}||_2^2 + K\sigma^2 \ell \Big\}  \\
    &= \Big\{ ||Y-X\Hat{\beta}_{m_r}||_2^2 - ||Y-X\Hat{\beta}_{m_\ell}||_2^2 < K\sigma^2(\ell-r) \Big\}.
\end{align*}
So, by applying Lemma~\ref{Lemma_1}, $\ell \in \{0,\cdots,r-1\}$~:
$$ \Big\{ \text{crit}_K(m_r) < \text{crit}_K(m_\ell) \Big\} = \Big\{ \underset{k=\ell+1}{\overset{r}{\sum}} \langle Y, u_k \rangle^2 > K\sigma^2(r-\ell) \Big\},$$
and by applying Lemma~\ref{Lemma_2}, $\ell \in \{r+1,\cdots,q\}$~: 
$$ \Big\{ \text{crit}_K(m_r) < \text{crit}_K(m_\ell) \Big\} = \Big\{ \underset{k=r+1}{\overset{\ell}{\sum}} \langle Y, u_k \rangle ^2 < K\sigma^2(\ell-r) \Big\}.$$
In this way, $\left\{\underset{\underset{\ell \neq r}{\ell = 0}}{\overset{q}{\cap}} \{ \text{crit}_K(m_r) < \text{crit}_K(m_\ell)\}\right\}$ is decomposed by two events~:
$$\left\{\underset{\ell = 0}{\overset{r-1}{\cap}} \left\{ \underset{k=\ell+1}{\overset{r}{\sum}} \langle Y, u_k \rangle^2 > K\sigma^2(r-\ell) \right\}\right\} \cap \left\{\underset{\ell = r+1}{\overset{q}{\cap}}\left\{ \underset{k=r+1}{\overset{\ell}{\sum}} \langle Y, u_k \rangle ^2 < K\sigma^2(\ell-r) \right\} \right\}.$$ 
Let us define $U$ the $n\times n$ matrix such that $u_k$ is the $k-$th column of $U$. Since $\varepsilon \sim \mathcal{N}(0,\sigma^2I_n)$ and $\left(u_1,\cdots,u_n\right)$ is an orthonormal basis of $\mathbb{R}^n$, we get $U^T \varepsilon = \Big(\langle \varepsilon, u_1 \rangle, \cdots, \langle \varepsilon, u_n \rangle \Big)^T \sim \mathcal{N}(0,\sigma^2 U I_n U^T) = \mathcal{N}(0,\sigma^2 I_n)$. Hence, random variables $\left(\langle Y, u_i \rangle \right)_{i \in \{1,\cdots,n\}}$ are independent with $\langle Y, u_i \rangle \sim \mathcal{N}\big(\langle X\beta^*,u_i \rangle,\sigma^2\big)$ for all $i$ in $\{1,\cdots,n\}$. Since the first event of the previous decomposition depends only on random variables $\langle Y, u_i \rangle$ for $i \in \{1,\cdots,r-1\}$ whereas the second one depends only on random variables $\langle Y, u_i \rangle$ for $i \in \{r+1,\cdots,q\}$, the two events are independent.
Hence, from~(\ref{FDR_with_crit}), we obtain for all $K>0$~: 
\begin{align*}
    \text{FDR}(\Hat{m}(K)) =  \underset{r = D_{m^*}+1}{\overset{q}{\sum}} \frac{r-D_{m^*}}{r} &\mathbb{P}\left( \underset{\ell = 0}{\overset{r-1}{\cap}} \left\{ \underset{k=\ell+1}{\overset{r}{\sum}} \langle Y,u_k \rangle^2 > K\sigma^2(r-\ell) \right\} \right) \\
    \times \ &\mathbb{P}\left( \underset{\ell = r+1}{\overset{q}{\cap}} \left\{ \underset{k=r+1}{\overset{\ell}{\sum}} \langle Y, u_k \rangle ^2 < K\sigma^2(\ell-r) \right\} \right).
\end{align*}

\vspace{0.2 cm}
Moreover, since $\langle X \beta^*, u_k \rangle = 0, \forall k > D_{m^*}$ and since $r \geq D_{m^*}+1$, we have~:
\begin{equation*}
    \underset{k=\ell+1}{\overset{r}{\sum}} \langle Y, u_k \rangle ^2 = \underset{k=\ell+1}{\overset{r}{\sum}} \langle \varepsilon, u_k \rangle^2.
\end{equation*}
So, for all $K>0$ and for each $r \in \{D_{m^*}+1, \cdots,q\}$~: 
\begin{align*}
    \mathbb{P}\left( \underset{\ell = r+1}{\overset{q}{\cap}} \left\{ \underset{k=r+1}{\overset{\ell}{\sum}} \langle Y, u_k \rangle ^2 < K\sigma^2(\ell-r) \right\} \right) &= \mathbb{P}\bigg( \underset{\ell = r+1}{\overset{q}{\cap}} \Big\{ \underset{k=r+1}{\overset{\ell}{\sum}} \Tilde{Z_k}^2 < K\sigma^2(\ell-r) \Big\} \bigg), \\
    &\hspace{-0.5 cm} \text{where} \ \ \Tilde{Z_k} \overset{i.i.d.}{\sim} \mathcal{N}(0,\sigma^2) \\
    \mathbb{P}\left( \underset{\ell = r+1}{\overset{q}{\cap}} \left\{ \underset{k=r+1}{\overset{\ell}{\sum}} \langle Y, u_k \rangle ^2 < K\sigma^2(\ell-r) \right\} \right) &= \mathbb{P}\bigg( \underset{\ell = r+1}{\overset{q}{\cap}} \Big\{ \underset{k=r+1}{\overset{\ell}{\sum}} Z_k^2 < K(\ell-r) \Big\} \bigg), \\
    &\hspace{-0.5 cm} \text{where} \ \ Z_k \overset{i.i.d.}{\sim} \mathcal{N}(0,1). \\
\end{align*}
Hence, for all $K>0$ and for each $r \in \{D_{m^*}+1, \cdots,q\}$, \\
$\mathbb{P}\left( \underset{\ell = r+1}{\overset{q}{\cap}} \left\{ \underset{k=r+1}{\overset{\ell}{\sum}} \langle Y, u_k \rangle ^2 < K\sigma^2(\ell-r) \right\} \right)$ does not depend on the data and we deduce the Formula~(\ref{FDR_1}) with~: 
\begin{align*}
        P_r(K) &= \mathbb{P}\bigg( \underset{\ell = r+1}{\overset{q}{\cap}} \Big\{ \underset{k=r+1}{\overset{\ell}{\sum}} Z_k^2 < K(\ell-r) \Big\} \bigg), \\
        Q_r(K,\beta^*,\sigma^2) &= \mathbb{P}\bigg( \underset{\ell = 0}{\overset{r-1}{\cap}} \Big\{ \underset{k=\ell+1}{\overset{r}{\sum}} \langle Y, u_k \rangle^2 > K\sigma^2(r-\ell) \Big\} \bigg),
\end{align*}
where $Z_k \overset{i.i.d.}{\sim} \mathcal{N}(0,1), \ \forall k \in \{r+1,\cdots,q\}$.
\qed

\subsection{General bounds.}
\label{proof_general_bounds}
\textbf{\textit{Proof of Theorem~\ref{theorem}}.}\\
We start from~(\ref{FDR_1}). 

\vspace{0.3 cm}
- \textbf{\textit{bounds on the $Q_r$ terms}.}\\
For all $K>0$ and for each $r \in \{D_{m^*}+1, \cdots,q\}$, we recall that~: 
\begin{equation*}
     Q_r(K,\beta^*,\sigma^2) = \mathbb{P}\left( \underset{\ell = 0}{\overset{r-1}{\cap}} \left\{ \underset{k=\ell+1}{\overset{r}{\sum}} \langle Y,u_k \rangle^2 > K\sigma^2(r-\ell) \right\} \right),
\end{equation*}
and since $\langle X \beta^*, u_k \rangle = 0, \forall k > D_{m^*}$, we have~:
\begin{align}
     &Q_r(K,\beta^*,\sigma^2) \notag \\ 
     &= \mathbb{P}\left( \underset{\ell = 0}{\overset{r-1}{\cap}} \left\{ \underset{k=\ell+1}{\overset{r}{\sum}} \Big( \langle \varepsilon, u_k \rangle^2 \ind{k > D_{m^*}} + \langle Y, u_k \rangle^2 \ind{k \leq D_{m^*}} \Big) > K\sigma^2(r-\ell) \right\} \right) \notag \\
    &= \mathbb{P}\bigg( \Big\{ \langle \varepsilon,u_r \rangle^2 > K \sigma^2 \Big\} \cap \cdots \cap \Big\{ \langle \varepsilon,u_r \rangle^2 + \cdots + \langle \varepsilon,u_{D_{m^*}+1} \rangle^2 > K \sigma^2 (r-D_{m^*}) \Big\} \notag \\
    & \cap \Big\{ \langle \varepsilon,u_r \rangle^2 + \cdots + \langle \varepsilon,u_{D_{m^*}+1} \rangle^2 + \langle Y,u_{D_{m^*}} \rangle^2 > K \sigma^2 (r-D_{m^*}+1) \Big\} \cap \cdots \notag \\
    & \cap \Big\{ \langle \varepsilon,u_r \rangle^2 + \cdots + \langle \varepsilon,u_{D_{m^*}+1} \rangle^2 + \langle Y,u_{D_{m^*}} \rangle^2 + \cdots + \langle Y,u_1 \rangle^2 > K \sigma^2 r \Big\} \bigg) \notag \\
    &= \mathbb{P} \Bigg(\Big\{ c_r > K \sigma^2 \Big\} \cap \Big\{ c_r + c_{r-1} > 2K \sigma^2 \Big\} \cap \cdots \cap \Big\{ c_r + c_{r-1} + \cdots + c_1  > rK \sigma^2 \Big\} \Bigg) 
    \label{complicated_P1}
\end{align}
where $c_\ell = \langle Y, u_\ell \rangle^2$ for $\ell \in \{1, \cdots D_{m^*}\}$ and $c_\ell = \langle \varepsilon, u_\ell \rangle^2$ for $\ell \in \{ D_{m^*}+1, \cdots, r \}$.

\vspace{0.2 cm}
\textbf{Lower bound on $Q_r(K,\beta^*,\sigma^2)$ for $r \in \{D_{m^*}+1, \cdots,q\}$}~:
\begin{lemma}
Let us consider an integer $s>1$, $K>0$ and $c_1, \cdots, c_s$ $s$ non-negative random independent quantities. We define by $E_\ell$ the event $\{ c_\ell > \ell K \sigma^2 \}$ for $\ell \in \{1, \cdots, s \}$ and by $F_\ell$ the event $\{ K \sigma^2 < c_\ell \leq \ell K \sigma^2 \}$ for $\ell \in \{2, \cdots, s\}$. \\
Then~: 
\begin{align*}
    & \Big\{c_s> K \sigma^2\Big\} \cap \Big\{c_s + c_{s-1} > 2K \sigma^2 \Big\} \cap \cdots \cap \Big\{c_s + c_{s-1} + \cdots + c_1 > s K \sigma^2 \Big\} \\
    & \supseteq E_s \sqcup \Bigg( F_s \sqcap \bigg(E_{s-1} \sqcup \Big( F_{s-1} \sqcap \big(E_{s-2} \sqcup \cdots \sqcup ( F_3 \sqcap (E_2 \sqcup (F_2 \sqcap E_1 ))) \big) \Big) \bigg) \Bigg),
\end{align*}
where $\cap$ and $\sqcap$ denote respectively any intersection and a disjoint intersection of events, as well as $\cup$ and $\sqcup$ denoting respectively any union and a disjoint union of events.
\label{Recurrence_lower}
\end{lemma}

\begin{proof}
We prove Lemma~\ref{Recurrence_lower} by a recurrence on $s \geq1 $. \\
For $s=1$, both sets correspond to $E_1$, so the inclusion is obvious. Let $s\geq 1$ and suppose that the inclusion is true for $s$. With the definitions of $E_{s+1}$ and $F_{s+1}$, we obtain~: 
\begin{align*}
    & \Big\{c_{s+1}> K \sigma^2\Big\} \cap \Big\{c_{s+1} + c_s > 2K \sigma^2 \Big\} \cap \cdots \cap \Big\{c_{s+1} + c_s + \cdots + c_1 > (s+1)K \sigma^2 \Big\} \\
    &= \Bigg( E_{s+1} \sqcup F_{s+1} \Bigg) \\
    & \cap \Bigg( \Big\{c_{s+1} + c_s > 2K \sigma^2 \Big\} \cap \cdots \cap \Big\{c_{s+1} + c_s + \cdots + c_1 > (s+1)K \sigma^2 \Big\} \Bigg) \\
    & = \Bigg( E_{s+1} \cap \bigg( \Big\{c_{s+1} + c_s > 2K \sigma^2 \Big\} \cap \cdots \cap \Big\{c_{s+1} + c_s + \cdots + c_1 > (s+1)K \sigma^2 \Big\} \bigg) \Bigg) \\
    & \sqcup \Bigg( F_{s+1} \cap \bigg( \Big\{c_{s+1} + c_s > 2K \sigma^2 \Big\} \cap \cdots \cap \Big\{c_{s+1} + c_s + \cdots + c_1 > (s+1)K \sigma^2 \Big\} \bigg) \Bigg) \\
    & \underset{\textit{(*)}}{=} E_{s+1}  \\
    & \sqcup \Bigg( F_{s+1} \cap \bigg(  \Big\{c_{s+1} + c_s > 2K \sigma^2 \Big\} \cap \cdots \cap \Big\{c_{s+1} + c_s + \cdots + c_1 > (s+1)K \sigma^2 \Big\} \bigg) \Bigg) \\
    & \underset{\textit{(**)}}{\supseteq}  E_{s+1}  \sqcup \Bigg( F_{s+1}  \\
    & \cap \bigg( \Big\{c_s > K \sigma^2 \Big\} \cap \Big\{ c_s + c_{s-1} > 2K \sigma^2\Big\} \cap \cdots \cap \{ c_s + c_{s-1} + \cdots + c_1 > sK \sigma^2 \} \bigg) \Bigg) \\
    & \underset{\textit{(***)}}{\supseteq}  E_{s+1} \sqcup \Bigg( F_{s+1} \cap \bigg( E_s \sqcup \Big( F_s \sqcap \big( E_{s-1} \sqcup \cdots \sqcup \left( F_3 \sqcap (E_3 \sqcup (F_2 \sqcap E_1)) \right) \big) \Big) \bigg) \Bigg) \\
    & \underset{\textit{(****)}}{\supseteq}  E_{s+1} \sqcup \Bigg( F_{s+1} \sqcap \bigg( E_s \sqcup \Big( F_s \sqcap \big(E_{s-1} \sqcup \cdots \sqcup \left( F_3 \sqcap (E_3 \sqcup (F_2 \sqcap E_1)) \right) \big) \Big) \bigg) \Bigg).
\end{align*}
\textit{(*)} is true since $c_i$ are non-negative for all $i \in \{1,\cdots,s+1\}$ providing that $E_{s+1} \subset \bigg( \Big\{c_{s+1} + c_s > 2K \sigma^2 \Big\} \cap \cdots \cap \Big\{c_{s+1} + c_s + \cdots + c_1 > (s+1)K \sigma^2 \Big\} \bigg)$, \textit{(**)} comes from the inclusion $\Big\{c_{s+1} > K \sigma^2 \Big\} \subset F_{s+1}$. We obtain \textit{(***)} by applying the recurrence assumption at the step $s$. Independence of $c_1,\cdots,c_{s+1}$ provides the independence between $F_{s+1}$ and $\bigg( E_s \sqcup \Big( F_s \sqcap \big( E_{s-1} \sqcup \cdots \sqcup \left( F_3 \sqcap (E_3 \sqcup (F_2 \sqcap E_1)) \right) \big) \Big) \bigg)$ which gets \textit{(****)}. \\
Thus, the property is true for $s+1$, which proves lemma. 
\end{proof}

\vspace{0.3 cm}
By applying Lemma~\ref{Recurrence_lower} on Formula~(\ref{complicated_P1}) with $s=r$, we obtain~:
\begin{align*}
    &Q_r(K,\beta^*,\sigma^2) \geq \mathbb{P}(E_r) \\
    & \quad + \mathbb{P}(F_r) \Bigg(  \mathbb{P}(E_{r-1}) +  \mathbb{P}(F_{r-1}) \bigg( \mathbb{P}(E_{r-2}) + \cdots + \mathbb{P}(F_3) \Big( \mathbb{P}(E_2) + \mathbb{P}(F_2)  \mathbb{P}(E_1) \Big) \bigg) \Bigg).
\end{align*}
By using that $\langle Y, u_\ell\rangle \sim \mathcal{N}(\langle X \beta^*, u_\ell \rangle,\sigma^2)$ for $\ell \in \{1,\cdots,D_{m^*}\}$ and $\langle \varepsilon, u_\ell\rangle \in \mathcal{N}(0,\sigma^2)$ for $\ell \in \{1,\cdots,r\}$, we get~: \\
For $\ \ell \in \{1, \cdots, D_{m^*}\}~:$
\begin{align*}
    \mathbb{P}(E_\ell)&= \mathbb{P}\bigg(\Big\{\langle Y, u_\ell\rangle^2 > \ell K \sigma^2\Big\}\bigg) \\
    &= 2 - \bigg( \Phi\Big(\sqrt{\ell K}-\frac{\langle X \beta^*, u_\ell \rangle}{\sigma}\Big) + \Phi\Big(\sqrt{\ell K}+\frac{\langle X \beta^*, u_\ell \rangle}{\sigma}\Big) \bigg) \\
    &= G_{\ell}.
\end{align*}
For $\ \ell \in \{2, \cdots, D_{m^*}\}~:$
\begin{align*}
    \mathbb{P}(F_\ell) &= \mathbb{P}\bigg(\Big\{K \sigma^2 < \langle Y, u_\ell \rangle^2 \leq \ell K \sigma^2\Big\}\bigg)\\
     &= \Phi\Big(\sqrt{\ell K}-\frac{\langle X \beta^*, u_\ell \rangle}{\sigma}\Big) + \Phi\Big(\sqrt{\ell K}+ \frac{\langle X \beta^*, u_\ell \rangle}{\sigma}\Big) \\
    & \quad \quad - \bigg(\Phi\Big(\sqrt{K}-\frac{\langle X \beta^*, u_\ell \rangle}{\sigma}\Big) + \Phi\Big(\sqrt{K}+\frac{\langle X \beta^*, u_\ell \rangle}{\sigma}\Big)\bigg) \\
     &= H_{\ell}.
\end{align*}
For $\ \ell \in \{D_{m^*}+1, \cdots, r\}~:$
\begin{align*}
    \mathbb{P}(E_\ell) &= \mathbb{P}\bigg(\Big\{\langle \varepsilon, u_\ell\rangle^2 > \ell K \sigma^2\Big\}\bigg) \\
    &= 2 \bigg(1-\Phi\big(\sqrt{\ell K}\big)\bigg)  \\
    &= G_{\ell}, \\
    \mathbb{P}(F_\ell) &= \mathbb{P}\bigg(\Big\{K \sigma^2 < \langle \varepsilon, u_\ell\rangle^2 \leq \ell K \sigma^2\Big\}\bigg) \\
    &= 2 \bigg( \Phi\big(\sqrt{\ell K}\big) - \Phi\big(\sqrt{K}\big)\bigg) \\
    &= H_{\ell}.
\end{align*}
Hence, a lower bound on $Q_r(K,\beta^*,\sigma^2)$ is obtained for all $K>0$~:
\begin{equation}
        \underline{f}_r(K,\beta^*,\sigma^2) \leq Q_r(K,\beta^*,\sigma^2)
            \label{first_term_lower}
\end{equation}
with~: 
\begin{align}
    \underline{f}_r(K,\beta^*,\sigma^2) &= G_{r} + H_{r} \ \underline{f}_{r-1}(K,\beta^*,\sigma^2) \notag \\
    &\text{and} \quad \underline{f}_1(K,\beta^*,\sigma^2) = G_1.
    \label{FDR_6}
\end{align}

\vspace{0.3 cm}

\vspace{0.32 cm}
\textbf{Upper bound on $Q_r(K,\beta^*,\sigma^2)$ for $r \in \{D_{m^*}+1, \cdots,q\}$}~: \\
By using definitions of Lemma~\ref{Recurrence_lower} and formula~(\ref{complicated_P1}), we get~:
\begin{align}
    Q_r(K,\beta^*,\sigma^2) 
    &\leq  \min \Bigg( \mathbb{P}\bigg(\Big\{c_r>K  \sigma^2\Big\}\bigg),\mathbb{P}\bigg(\Big\{ c_r + c_{r-1} > 2K \sigma^2\Big\}\bigg), \cdots,  \notag \\
    &\quad \quad \quad \quad \mathbb{P} \bigg(\Big\{ c_r+c_{r-1}+\cdots +c_1 > rK \sigma^2 \Big\}\bigg) \Bigg).
    \label{FDR_4}
\end{align}
Since $\langle \varepsilon, u_i \rangle_{i \in \{D_{m^*}+1,\cdots,r\}} \overset{\text{i.i.d.}}{\sim} \mathcal{N}\big(0,\sigma^2\big)$, we have $\text{for all} \ j \in \{D_{m^*}+1,\cdots, r\}$~:
\begin{equation}
   \mathbb{P}\bigg(\Big\{ c_r + \cdots + c_j > (r-j+1)K \sigma^2\Big\}\bigg) =  1 - F_{\chi^2(r-j+1)}\Big((r-j+1)K\Big).
   \label{temp_1}
\end{equation}

\vspace{0.4 cm}
For all $j \in \{1,\cdots,D_{m^*}\}$, 
\begin{align*}
    &\mathbb{P}\bigg(\Big\{ c_r + \cdots + c_j > (r-j+1)K \sigma^2\Big\}\bigg) \\
    &= \mathbb{P}\bigg(\Big\{ c_r + \cdots + c_{D_{m^*}+1} + c_{D_{m^*}} + \cdots + c_j > (r-j+1)K \sigma^2\Big\}\bigg)  \\
    &= \mathbb{P}\bigg(\Big\{  c_r + \cdots + c_{D_{m^*}+1} + \big( \langle X\beta^*,u_{D_{m^*}} \rangle + \langle \varepsilon, u_{D_{m^*}} \rangle \big)^2 + \cdots \\
    & \hspace{4 cm} + \big( \langle X\beta^*,u_{j} \rangle + \langle \varepsilon, u_{j} \rangle\big)^2 > (r-j+1)K \sigma^2\Big\}\bigg)   \\
    &\underset{\textit{(**)}}{\leq} \mathbb{P}\bigg(\Big\{ c_r + \cdots + c_{D_{m^*}+1} + 2\langle X\beta^*,u_{D_{m^*}} \rangle^2 + 2 \langle \varepsilon, u_{D_{m^*}} \rangle^2 + \cdots  \\
    &\hspace{4 cm} + 2\langle X\beta^*,u_{j} \rangle^2 + 2 \langle \varepsilon, u_{j} \rangle^2 > (r-j+1)K \sigma^2\Big\}\bigg)  \\
    &\leq \mathbb{P}\bigg(\Big\{ 2 c_r + \cdots + 2 c_{D_{m^*}+1} + 2 \langle \varepsilon, u_{D_{m^*}} \rangle^2 + \cdots + 2 \langle \varepsilon, u_{j} \rangle^2 > (r-j+1)K \sigma^2   \\
    &\hspace{4 cm} - 2\langle X\beta^*,u_{D_{m^*}} \rangle^2 - \cdots - 2\langle X\beta^*,u_{j} \rangle^2 \Big\}\bigg)\\
    &\underset{\textit{(***)}}{=} \mathbb{P}\bigg(\Big\{ 2 \sigma^2 Z_r^2 + \cdots + 2 \sigma^2 Z_{D_{m^*}+1}^2 + 2 \sigma^2 Z_{D_{m^*}}^2 + \cdots + 2 \sigma^2 Z_{j}^2  \\
    & \hspace{3 cm} > (r-j+1)K \sigma^2 - 2\langle X\beta^*,u_{D_{m^*}} \rangle^2 - \cdots - 2\langle X\beta^*,u_{j} \rangle^2 \Big\}\bigg),  \\
    & \hspace{2 cm} \text{where} \ (Z_\ell)_{\ell \in \{j,\cdots,r\}} \overset{i.i.d}{\sim} \mathcal{N}(0,1). \\
    & = \mathbb{P}\bigg(\Big\{ Z_r^2 + \cdots + Z_{D_{m^*}+1}^2 + Z_{D_{m^*}}^2 + \cdots + Z_{j}^2  \\
    & \hspace{3 cm} > \frac{(r-j+1)K}{2} - \frac{\langle X\beta^*,u_{D_{m^*}} \rangle^2}{\sigma^2} - \cdots - \frac{\langle X\beta^*,u_{j} \rangle^2}{\sigma^2} \Big\}\bigg) \\
\end{align*}
\begin{align}
    & = \mathbb{P}\bigg(\Big\{ X > \frac{(r-j+1)K}{2} - \frac{\langle X\beta^*,u_{D_{m^*}} \rangle^2}{\sigma^2} - \cdots - \frac{\langle X\beta^*,u_{j} \rangle^2}{\sigma^2} \Big\}\bigg), \notag \\
    & \hspace{2 cm} \text{for} \ X \sim \chi^2(r-j+1) \notag  \\
    & = 1 - F_{\chi^2(r-j+1)}\bigg(\frac{(r-j+1)K}{2} - \frac{\langle X\beta^*,u_{D_{m^*}} \rangle^2}{\sigma^2} - \cdots - \frac{\langle X\beta^*,u_{j} \rangle^2}{\sigma^2} \bigg).
   \label{temp_2}
\end{align}
\textit{(**)} provides from $(a+b)^2 \leq 2(a^2 + b^2), \ \forall (a,b) \in \mathbb{R}$ and \textit{(***)} is true since $\langle \varepsilon, u_i \rangle_{i \in \{1,\cdots,r\}} \overset{\text{i.i.d.}}{\sim} \mathcal{N}\big(0,\sigma^2\big)$. \\

So, from~(\ref{FDR_4}),~(\ref{temp_1}) and ~(\ref{temp_2}), we deduce that for all $K>0$ and for each $r \in \{D_{m^*}+1, \cdots,q\}$~:
\begin{align*}
    &Q_r(K,\beta^*,\sigma^2) \leq \\
    &\quad \quad \min \Bigg( 1- F_{\chi^2(1)}(K), \cdots, 1- F_{\chi^2(r-D_{m^*})}\big((r-D_{m^*})K\big), \\
    &\quad \quad 1- F_{\chi^2(r-D_{m^*}+1)}\bigg(\frac{(r-D_{m^*}+1)K}{2} - \frac{\langle X\beta^*,u_{D_{m^*}} \rangle^2}{\sigma^2}\bigg), \\
    &\quad \quad 1- F_{\chi^2(r-D_{m^*}+2)}\bigg(\frac{(r-D_{m^*}+2)K}{2} - \frac{\langle X\beta^*,u_{D_{m^*}} \rangle^2}{\sigma^2} - \frac{\langle X\beta^*,u_{D_{m^*-1}} \rangle^2}{\sigma^2}\bigg), \\
    &\quad \quad  \cdots, \\
    &\quad \quad 1- F_{\chi^2(r)}\bigg(\frac{rK}{2} - \frac{\langle X\beta^*,u_{D_{m^*}}\rangle^2}{\sigma^2} - \frac{\langle X\beta^*,u_{D_{m^*-1}} \rangle^2}{\sigma^2} - \cdots - \frac{\langle X\beta^*,u_1 \rangle^2}{\sigma^2}\bigg)\Bigg).
\end{align*}
Hence, an upper bound on $Q_r(K,\beta^*,\sigma^2)$ is obtained for all $K>0$~:
\begin{equation}
    Q_r(K,\beta^*,\sigma^2) \leq \overline{f}_r(K,\beta^*,\sigma^2))
    \label{first_term_upper}
\end{equation}
with~: 
\begin{align}
    \overline{f}_r(K,\beta^*,\sigma^2) = 1-\max\Bigg(&\underset{\ell \in \{1,\cdots,r-D_{m^*}\}}{\max}\Big(F_{\chi^2(\ell)}(\ell K)\Big), \notag \\
    &\underset{\ell \in \{r-D_{m^*}+1,\cdots,r\}}{\max}\bigg(F_{\chi^2(\ell)}\Big(\frac{\ell K}{2} - \underset{k=r-\ell+1}{\overset{D_{m^*}}{\sum}}\frac{\langle X\beta^*,u_k \rangle^2}{\sigma^2}\Big)\bigg)\Bigg).
    \label{FDR_3}
\end{align}

- \textbf{\textit{bounds on the FDR}.}\\
By combining~(\ref{FDR_1}),~(\ref{first_term_lower}),~(\ref{FDR_6}),~(\ref{first_term_upper}),~(\ref{FDR_3}) and~(\ref{second_term}), we obtain~:
\begin{equation*}
    \sum\limits_{r = D_{m^*}+1}^q \Bigg(\frac{r-D_{m^*}}{r} P_r(K) \underline{f}_r(K,\beta^*,\sigma^2) \Bigg) \leq \text{FDR}(\Hat{m}(K))
\end{equation*}
and
\begin{equation*}
    \text{FDR}(\Hat{m}(K))
    \leq \sum\limits_{r = D_{m^*}+1}^q \Bigg( \frac{r-D_{m^*}}{r} P_r(K) \overline{f}_r(K,\beta^*,\sigma^2) \Bigg),
\end{equation*}
which allows us to obtain Theorem~\ref{theorem} with $\forall K>0$,
\begin{equation*}
    b(K,\beta^*,\sigma^2) = \sum\limits_{r = D_{m^*}+1}^q \Bigg(\frac{r-D_{m^*}}{r} P_r(K)  \underline{f}_r(K,\beta^*,\sigma^2)\Bigg)
\end{equation*}
and 
\begin{equation*}
    B(K,\beta^*,\sigma^2) = \sum\limits_{r = D_{m^*}+1}^q \Bigg( \frac{r-D_{m^*}}{r} P_r(K) \overline{f}_r(K,\beta^*,\sigma^2) \Bigg).
\end{equation*}
\qed

\subsection{Strictly positive FDR.}
\label{proof_no_noise}
\textbf{\textit{Proof of Corollary~\ref{positive_FDR}}.}\\
From Theorem~\ref{theorem}, we have $\forall K>0$,
\begin{equation}
    \text{FDR}(\Hat{m}(K)) \geq \sum\limits_{r = D_{m^*}+1}^q \Bigg(\frac{r-D_{m^*}}{r} P_r(K) \ \underline{f}_r(K,\beta^*,\sigma^2) \Bigg).
    \label{lower_bound_explicite}
\end{equation}
For the rest of the proof, we use the following Lemma~:
\begin{lemma}[Frank R. Kschischang \cite{kschischang2017complementary}]
The complementary error function, $\text{erfc}(x)$, is defined, for $x \geq 0$, as~:
\begin{equation*}
    \text{erfc}(x) = 2 \Big( 1- F_{\mathcal{N}(0,\frac{1}{2})}(x)\Big)
\end{equation*}
where $F_{\mathcal{N}(0,\frac{1}{2})}$ denotes the cumulative function of the centered Gaussian with the variance equals $\frac{1}{2}$. \\
Then, 
\begin{equation*}
    \forall x \geq 0, \quad \quad \frac{2 e^{-x^2}}{\sqrt{\pi}\big(x+\sqrt{x^2+2}\big)} \leq \text{erfc}(x) \leq \frac{e^{-x^2}}{\sqrt{\pi}x}.
\end{equation*}
\label{erfc}
\end{lemma}

\vspace{0.5 cm}
We remark that for all $x \geq 0$, $1-\Phi(x) = \frac{1}{2} \text{erfc}\big(\frac{x}{\sqrt{2}}\big)$. Then, for each $r\in\{D_{m^*}+1,\cdots,q\}$, 
\begin{align}
    \underline{f}_r(K,\beta^*,\sigma^2) &= G_{r} + H_{r} \Big( G_{r-1} +  H_{r-1} \big( G_{r-2} +  \cdots + H_{2}  G_{1} \big) \Big) \notag \\
    & \geq G_{r} \notag \\
    & = 2 \bigg(1-\Phi\big(\sqrt{r K}\big)\bigg) \notag \\
    & = \text{ercf}\Big(\sqrt{\frac{rK}{2}}\Big) \notag \\
    & \underset{\textit{(**)}}{\geq} \frac{2}{\sqrt{\pi} \Big(\sqrt{\frac{rK}{2}}+\sqrt{\frac{rK}{2}+2}\Big)}  e^{-\frac{rK}{2}} \notag \\
    & = \frac{2\sqrt{2}}{\sqrt{\pi}\Big(\sqrt{rK}+\sqrt{rK+4}\Big)} e^{-\frac{rK}{2}}.
    \label{CV_3}
\end{align}
\textit{(**)} is provided by Lemma \ref{erfc}. So, from~(\ref{lower_bound_explicite}) and~(\ref{CV_3}), we obtain~: 
\begin{equation*}
    \forall K > 0, \ \ \text{FDR}(\Hat{m}(K)) \geq \sum\limits_{r = D_{m^*}+1}^q \Bigg(\frac{r-D_{m^*}}{r} P_r(K) \frac{2\sqrt{2}}{\sqrt{\pi}\Big(\sqrt{rK}+\sqrt{rK+4}\Big)} e^{-\frac{rK}{2}} \Bigg).
\end{equation*}
This lower bound is strictly positive and since the $P_r(K)$ terms are all strictly positive too, we deduce that the $\text{FDR}$ function is a strictly positive function.
\qed

\vspace{1 cm}
\subsection{Asymptotic analysis.}
\label{proof_asymptotic}
\textbf{\textit{Proof of Corollary~\ref{asympt_K}}.}\\
For all $r \in \{D_{m^*} +1, \cdots, q\}$ and by using the definitions from Theorem~\ref{theorem}, \\
for $\ \ell \in \{1, \cdots, D_{m^*}\}~:$
\begin{align*}
    &G_{\ell} = 2 - \bigg( \Phi\Big(\sqrt{\ell K}-\frac{\langle X \beta^*, u_\ell \rangle}{\sigma} \rangle\Big) + \Phi\Big(\sqrt{\ell K}+\frac{\langle X \beta^*, u_\ell \rangle}{\sigma}\Big) \bigg) \\ 
    &\quad \quad \quad \quad \quad \underset{K \longrightarrow + \infty}{\longrightarrow}  0; 
\end{align*}   
for $\ \ell \in \{2, \cdots, D_{m^*}\}~:$
\begin{align*}
    &H_{\ell} = \Phi\Big(\sqrt{\ell K}-\frac{\langle X \beta^*, u_\ell \rangle}{\sigma}\Big) + \Phi\Big(\sqrt{\ell K}+ \frac{\langle X \beta^*, u_\ell \rangle}{\sigma}\Big) - \\
    &\quad \quad \quad \quad \bigg(\Phi\Big(\sqrt{K}-\frac{\langle X \beta^*, u_\ell \rangle}{\sigma}\Big) + \Phi\Big(\sqrt{K}+\frac{\langle X \beta^*, u_\ell \rangle}{\sigma}\Big)\bigg) \\
    &\quad \quad \quad \quad \quad \underset{K \longrightarrow + \infty}{\longrightarrow}  0;
\end{align*}
and for $ \ \ell \in \{D_{m^*}+1, \cdots, r\}~: $
\begin{align*}
     &G_{\ell} = 2 \bigg(1-\Phi\big(\sqrt{\ell K}\big)\bigg) \underset{K \longrightarrow + \infty}{\longrightarrow} 0, \\
     &H_{\ell} = 2 \bigg( \Phi\big(\sqrt{\ell K}\big) - \Phi\big(\sqrt{K}\big)\bigg) \underset{K \longrightarrow + \infty}{\longrightarrow}  0;
\end{align*}
which provides that $\underline{f}_r(K,\beta^*,\sigma^2) \underset{K \longrightarrow + \infty}{\longrightarrow} 0$. \\
Moreover, $\overline{f}_r(K,\beta^*,\sigma^2)) \underset{K \longrightarrow + \infty}{\longrightarrow} 0$. So, $Q_r(K,\beta^*,\sigma^2) \underset{K \longrightarrow + \infty}{\longrightarrow} 0$. 
In the same way, $P_r(K) \underset{K \longrightarrow + \infty}{\longrightarrow} 1$. So, $P_r(K) Q_r(K,\beta^*,\sigma^2)  \underset{K \longrightarrow + \infty}{\longrightarrow} 0$. \\
Finally, for each $r \in \{D_{m^*} +1, \cdots, q\}$, we deduce from~(\ref{FDR_1}) that 
\begin{equation*}
    \text{FDR}(\Hat{m}(K)) \underset{K \longrightarrow + \infty}{\longrightarrow} 0.
\end{equation*}

\vspace{0.3 cm}
For each $r \in \{D_{m^*}+1,\cdots,q\}$ $P_r(K) \underset{K \longrightarrow + \infty}{\longrightarrow}  1$, we deduce that for all $C_1 \in ]0,1[$, there exists $\Tilde{L}_{C_1}>0$ such that $\forall K > \Tilde{L}_{C_1}$ and $\forall r \in \{D_{m^*}+1,\cdots,q\}$, we have $C_1 \leq P_r(K)$. For the following, we fix $C_1 \in ]0,1[$. \\
By using~(\ref{first_term_lower}),~(\ref{first_term_upper}) and $P_r(K) \leq 1$  for each $r \in \{D_{m^*}+1,\cdots,q\}$, we deduce that~:
\begin{equation}
    \forall K > \Tilde{L}_{C_1}, \quad \quad \text{FDR}(\Hat{m}(K)) \geq C_1 \sum\limits_{r = D_{m^*}+1}^q \Bigg(\frac{r-D_{m^*}}{r} \underline{f}_r(K,\beta^*,\sigma^2) \Bigg)
    \label{CV_1}
\end{equation}
and 
\begin{equation}
    \forall K >0, \quad \quad \text{FDR}(\Hat{m}(K)) \leq \sum\limits_{r = D_{m^*}+1}^q \Bigg( \frac{r-D_{m^*}}{r} \overline{f}_r(K,\beta^*,\sigma^2)) \Bigg).
    \label{CV_2}
\end{equation}

\vspace{1 cm}
- \textbf{Upper bound on $\overline{f}_r$}~: \\
For each $r\in\{D_{m^*}+1,\cdots,q\}$ and for all $K > 0$~:
\begin{align}
        \overline{f}_r(K,\beta^*,\sigma^2)) &= 1-\max\Bigg(\underset{\ell \in \{1,\cdots,r-D_{m^*}\}}{\max}\Big(F_{\chi^2(\ell)}(\ell K)\Big), \notag \\
        & \hspace{1 cm} \underset{\ell \in \{r-D_{m^*}+1,\cdots,r\}}{\max}\bigg(F_{\chi^2(\ell)}\Big(\frac{\ell K}{2} - \underset{k=r-\ell+1}{\overset{D_{m^*}}{\sum}}\frac{\langle X\beta^*,u_k \rangle^2}{\sigma^2}\Big)\bigg)\Bigg) \notag \\
        &= \min\Bigg(\underset{\ell \in \{1,\cdots,r-D_{m^*}\}}{\min}\Big(\mathbb{P}\big(X_\ell > \ell K\big)\Big), \notag \\
        & \hspace{0.4 cm} \underset{\ell \in \{r-D_{m^*}+1,\cdots,r\}}{\min}\bigg(\mathbb{P}\big(Y_\ell > \frac{\ell K}{2} - \underset{k=r-\ell+1}{\overset{D_{m^*}}{\sum}}\frac{\langle X\beta^*,u_k \rangle^2}{\sigma^2}\big)\bigg)\Bigg), \notag \\
        & \hspace{1.8 cm} \text{with} \ X_\ell \sim \chi^2(\ell) \ \text{and} \ Y_\ell \sim \chi^2(\ell) \notag \\
        &= \min\Bigg(\underset{\ell \in \{1,\cdots,r-D_{m^*}\}}{\min}\Big(\mathbb{P}\big(X_\ell - \ell > \ell K - \ell \big)\Big), \notag \\
        & \hspace{0.4 cm} \underset{\ell \in \{r-D_{m^*}+1,\cdots,r\}}{\min}\bigg(\mathbb{P}\Big(Y_\ell - \ell > \frac{\ell (K-2)}{2} - \underset{k=r-\ell+1}{\overset{D_{m^*}}{\sum}}\frac{\langle X\beta^*,u_k \rangle^2}{\sigma^2}\Big)\bigg)\Bigg), \notag \\
        & \hspace{1.8 cm} \text{with} \ X_\ell \sim \chi^2(\ell) \ \text{and} \ Y_\ell \sim \chi^2(\ell). 
        \label{finess_term}
\end{align}

\vspace{0.3 cm}
So, for each $r\in\{D_{m^*}+1,\cdots,q\}$ and for all $K > 0$~:
\begin{equation*}
    \overline{f}_r(K,\beta^*,\sigma^2)) \leq \underset{\ell \in \{1,\cdots,r-D_{m^*}\}}{\min}\Big(\mathbb{P}\big(X_\ell - \ell > \ell K - \ell \big)\Big), \quad \quad \text{with}  \ X_\ell \sim \chi^2(\ell).
\end{equation*}

By the exponential inequality of \cite{laurent2000adaptive} for $X \sim \chi^2(\ell)$ and $\ell \in \mathbb{N}^{*}$~:  
\begin{equation}
    \forall x \geq 0, \quad \mathbb{P}\Big(X - \ell > 2 \sqrt{\ell x} + 2x \Big) \leq e^{-x}.
    \label{exponential_inequality_chi}
\end{equation}
We apply~(\ref{exponential_inequality_chi}) for each $\ell=1,\cdots,(r-D_{m^*})$ with $x = \frac{\ell}{4} \Big(1 - \sqrt{2K - 1}\Big)^2$ which is one solution of $ \ 2 \sqrt{\ell x} + 2x = \ell K - \ell \ $ when $K>1$. We obtain for all $K>1$~:
\begin{align}
    \underset{\ell \in \{1,\cdots,r-D_{m^*}\}}{\min}\Big(\mathbb{P}\big(X_\ell - \ell > \ell K - \ell \big)\Big) &\leq \underset{\ell=1,\cdots,(r-D_{m^*})}{\min} \bigg( e^{-\frac{\ell}{4} \big(1 - \sqrt{2K - 1}\big)^2} \bigg) \notag \\
    &\leq e^{\frac{(r-D_{m^*})\sqrt{2K-1}}{2}} e^{-\frac{(r-D_{m^*}) K}{2}}.
    \label{CV_4}
\end{align}

So, from~(\ref{CV_2}) and~(\ref{CV_4}), we obtain for each $r\in\{D_{m^*}+1,\cdots,q\}$ and for all $K>1$~: 
\begin{align*}
    \text{FDR}(\Hat{m}(K)) & \leq \sum\limits_{r = D_{m^*}+1}^q \Bigg( \frac{r-D_{m^*}}{r} e^{\frac{(r-D_{m^*})\sqrt{2K-1}}{2}} e^{-\frac{(r-D_{m^*}) K}{2}} \Bigg) \\
    & \leq e^{-\frac{K}{2}} \sum\limits_{r = D_{m^*}+1}^q \Bigg( \frac{r-D_{m^*}}{r} e^{\frac{(r-D_{m^*})\sqrt{2K-1}}{2}} \Bigg).
\end{align*}
For all $\eta>0$ and $r\in\{D_{m^*}+1,\cdots,q\}$, $e^{\frac{(r-D_{m^*})\sqrt{2K-1}}{2}} = \underset{K \longrightarrow + \infty}{o}\Big(e^{\eta K}\Big)$. \\
Hence, $\forall \eta>0$
\begin{equation*}
    \text{FDR}(\Hat{m}(K)) = \underset{K \longrightarrow + \infty}{o}\Big(e^{-K(\frac{1}{2}-\eta)}\Big),
\end{equation*}
which allows to obtain~(\ref{CV_rate_K_infinity}). \\

\vspace{0.5 cm}
\textbf{Proof of Remark~\ref{remark_finest}}~: \\
The inequalities~(\ref{CV_1}) and~(\ref{CV_2}) are also true when $K \longrightarrow + \infty$ and $\sigma \longrightarrow 0 \ \text{with} \ \frac{1}{\sigma} = \underset{\sigma \longrightarrow 0}{o}(\sqrt{K})$. To obtain the finest asymptotic upper bound~(\ref{CV_rate_K_infinity_more_complex}), we start from the equation~(\ref{finess_term}) and we consider the second term. Similar to previously, we apply~(\ref{exponential_inequality_chi}) for each $\ell=r-D_{m^*}+1,\cdots,r$ with 
$$x = \frac{\ell}{4} \Bigg(1-\sqrt{K-1 - \frac{2}{\ell} \underset{k=r-\ell+1}{\overset{D_{m^*}}{\sum}}\frac{\langle X\beta^*,u_k \rangle^2}{\sigma^2}}\Bigg)^2,$$
which is one solution of 
$$ 2 \sqrt{\ell x} + 2x = \frac{\ell (K-2)}{2} - \underset{k=r-\ell+1}{\overset{D_{m^*}}{\sum}}\frac{\langle X\beta^*,u_k \rangle^2}{\sigma^2} $$ when $ \ \sigma^2(K-1)>\frac{2}{r-D_{m^*}+1} \underset{k=1}{\overset{D_{m^*}}{\sum}}\langle X\beta^*,u_k \rangle^2 +2$. This condition is valid since $\sigma \longrightarrow 0$ with $\frac{1}{\sigma} = \underset{\sigma \longrightarrow 0}{o}(\sqrt{K})$ leading to $\frac{1}{\sigma^2} = \underset{\sigma \longrightarrow 0}{o}(K)$ and so $\sigma^2(K-1) \longrightarrow + \infty$ when $K \longrightarrow + \infty$. We obtain for all $K>0$ such that $\sigma^2(K-1)>\frac{2}{r-D_{m^*}+1} \underset{k=1}{\overset{D_{m^*}}{\sum}}\langle X\beta^*,u_k \rangle^2+2$~:
\begin{align}
    \underset{\ell \in \{r-D_{m^*}+1,\cdots,r\}}{\min} &\bigg(\mathbb{P}\big(Y_\ell - \ell > \frac{\ell (K-2)}{2} - \underset{k=r-\ell+1}{\overset{D_{m^*}}{\sum}}\frac{\langle X\beta^*,u_k \rangle^2}{\sigma^2}\big)\bigg)\Bigg) \notag \\
    &\leq \underset{\ell \in \{r-D_{m^*}+1,\cdots,r\}}{\min}\bigg( e^{-\frac{\ell}{4} \Bigg(1-\sqrt{K-1 - \frac{2}{\ell} \underset{k=r-\ell+1}{\overset{D_{m^*}}{\sum}}\frac{\langle X\beta^*,u_k \rangle^2}{\sigma^2}}\Bigg)^2} \bigg) \notag \\
    & \leq e^{\frac{1}{2} \underset{k=1}{\overset{D_{m^*}}{\sum}}\frac{\langle X\beta^*,u_k \rangle^2}{\sigma^2}} e^{\frac{r}{2} \sqrt{K-1-\frac{2}{r} \underset{k=1}{\overset{D_{m^*}}{\sum}}\frac{\langle X\beta^*,u_k \rangle^2}{\sigma^2}}} e^{-\frac{rK}{4}}.
    \label{CV_5}
\end{align}
\textit{(*)} come from the fact that a minimum into a set is smaller than any value in the set. We choose the value corresponding for $\ell = 0$. \\
So, from~(\ref{CV_2}),~(\ref{CV_4}) and~(\ref{CV_5}), we obtain for each $r\in\{D_{m^*}+1,\cdots,q\}$ and for all $K>1$ respecting $ \ \sigma^2(K-1)>\frac{2}{r-D_{m^*}+1} \underset{k=1}{\overset{D_{m^*}}{\sum}}\langle X\beta^*,u_k \rangle^2 +2$~: 
\begin{align}
    \text{FDR}(\Hat{m}(K)) & \leq \sum\limits_{r = D_{m^*}+1}^q \Bigg( \frac{r-D_{m^*}}{r} \min \bigg( e^{\frac{(r-D_{m^*})\sqrt{2K-1}}{2}} e^{-\frac{(r-D_{m^*}) K}{2}}, \notag \\
    &\hspace{2 cm} e^{\frac{1}{2} \underset{k=1}{\overset{D_{m^*}}{\sum}}\frac{\langle X\beta^*,u_k \rangle^2}{\sigma^2}} e^{\frac{r}{2} \sqrt{K-1-\frac{2}{r} \underset{k=1}{\overset{D_{m^*}}{\sum}}\frac{\langle X\beta^*,u_k \rangle^2}{\sigma^2}}} e^{-\frac{rK}{4}} \bigg) \Bigg) \notag \\
    & = \min \Bigg( \sum\limits_{r = D_{m^*}+1}^q \bigg( \frac{r-D_{m^*}}{r} e^{\frac{(r-D_{m^*})\sqrt{2K-1}}{2}} e^{-\frac{(r-D_{m^*}) K}{2}}\bigg),\notag \\
    & \sum\limits_{r = D_{m^*}+1}^q \bigg( \frac{r-D_{m^*}}{r} e^{\frac{1}{2} \underset{k=1}{\overset{D_{m^*}}{\sum}}\frac{\langle X\beta^*,u_k \rangle^2}{\sigma^2}} e^{\frac{r}{2} \sqrt{K-1-\frac{2}{r} \underset{k=1}{\overset{D_{m^*}}{\sum}}\frac{\langle X\beta^*,u_k \rangle^2}{\sigma^2}}} e^{-\frac{rK}{4}} \bigg) \Bigg) \notag \\
    & \leq \min \Bigg( e^{-\frac{K}{2}} \sum\limits_{r = D_{m^*}+1}^q \Bigg( \frac{r-D_{m^*}}{r} e^{\frac{(r-D_{m^*})\sqrt{2K-1}}{2}} \Bigg),\notag \\
    \sum\limits_{r = D_{m^*}+1}^q \bigg( &\frac{r-D_{m^*}}{r} e^{\frac{r}{2} \sqrt{K-1-\frac{2}{r} \underset{k=1}{\overset{D_{m^*}}{\sum}}\frac{\langle X\beta^*,u_k \rangle^2}{\sigma^2}}}\bigg) \ \ e^{-\big( \frac{(D_{m^*}+1)K}{4} - \frac{1}{2\sigma^2}  \underset{k=1}{\overset{D_{m^*}}{\sum}}\langle X\beta^*,u_k \rangle^2\big)} \Bigg).
    \label{CV_7}
\end{align}
For all $\eta>0$ and $r\in\{D_{m^*}+1,\cdots,q\}$, $e^{\frac{(r-D_{m^*})\sqrt{2K-1}}{2}} = \underset{K \longrightarrow + \infty}{o}\Big(e^{\eta K}\Big)$, independently of the value of $\sigma^2$. Hence, the first term in~(\ref{CV_7}) is $o\Big(e^{-K(\frac{1}{2}-\eta)}\Big), \forall \eta>0$ when $K \longrightarrow + \infty$ and $\sigma \longrightarrow 0 \ \text{with} \ \frac{1}{\sigma} = \underset{\sigma \longrightarrow 0}{o}(\sqrt{K})$. \\
For all $r\in\{D_{m^*}+1,\cdots,q\}$, $e^{\frac{r}{2} \sqrt{K-1-\frac{2}{r} \underset{k=1}{\overset{D_{m^*}}{\sum}}\frac{\langle X\beta^*,u_k \rangle^2}{\sigma^2}}} \leq e^{\frac{r}{2}\sqrt{K}}$. Moreover, for all $\Tilde{\eta} >0$ and $r\in\{D_{m^*}+1,\cdots,q\}$, $e^{\frac{r}{2}\sqrt{K}} = \underset{K \longrightarrow + \infty}{o}\Big(e^{\Tilde{\eta} K}\Big)$, independently of the value of $\sigma^2$. Hence, the second term in~(\ref{CV_7}) is $o\bigg(e^{- \Big(K\frac{(D_{m^*}+1-\Tilde{\eta})}{4} - \frac{1}{2\sigma^2} \underset{k=1}{\overset{D_{m^*}}{\sum}}\langle X\beta^*,u_k \rangle^2\Big)}\bigg)$, $\forall \Tilde{\eta}>0$ when $K \longrightarrow + \infty$ and $\sigma \longrightarrow 0 \ \text{with} \ \frac{1}{\sigma} = \underset{\sigma \longrightarrow 0}{o}(\sqrt{K})$. \\
Hence, 
\begin{align*}
    \text{FDR}(\Hat{m}(K)) &\leq \min \Bigg( o\Big(e^{-K(\frac{1}{2}-\eta)}\Big), o\bigg(e^{- \Big(K\frac{(D_{m^*}+1-\Tilde{\eta})}{4} - \frac{1}{2\sigma^2} \underset{k=1}{\overset{D_{m^*}}{\sum}}\langle X\beta^*,u_k \rangle^2\Big)}\bigg) \Bigg) \\
    &= o\bigg(e^{- \Big(K\frac{(D_{m^*}+1-\Tilde{\eta})}{4} - \frac{1}{2\sigma^2} \underset{k=1}{\overset{D_{m^*}}{\sum}}\langle X\beta^*,u_k \rangle^2\Big)}\bigg).
\end{align*}
$\forall (\eta,\Tilde{\eta})>0$ when $K \longrightarrow + \infty$ and $\sigma \longrightarrow 0 \ \text{with} \ \frac{1}{\sigma} = \underset{\sigma \longrightarrow 0}{o}(\sqrt{K})$; which allows us to obtain~(\ref{CV_rate_K_infinity_more_complex}). 

\vspace{1 cm}
- \textbf{Lower bound on $\underline{f}_r$}~: \\
From~(\ref{CV_3}) and~(\ref{CV_1}), we obtain~:
\begin{align*}
    \forall K > \Tilde{L}_{C_1}, \ \text{FDR}(\Hat{m}(K)) & \geq C_1 \sum\limits_{r = D_{m^*}+1}^q \Bigg(\frac{r-D_{m^*}}{r} \frac{2\sqrt{2}}{\sqrt{\pi}\Big(\sqrt{rK}+\sqrt{rK+4}\Big)} e^{-\frac{rK}{2}} \Bigg) \notag \\
    & \geq C_1 \frac{2\sqrt{2}}{\sqrt{\pi}\Big(\sqrt{qK}+\sqrt{qK+4}\Big)} \frac{1}{D_{m^*}+1} \sum\limits_{r = D_{m^*}+1}^q \Bigg( e^{-\frac{rK}{2}} \Bigg) \notag \\
    & \underset{\textit{(*)}}{\geq} C_1 \frac{2\sqrt{2}}{\sqrt{\pi}\Big(\sqrt{qK}+\sqrt{qK+4}\Big)} \frac{1}{D_{m^*}+1} e^{-\frac{(D_{m^*}+1)K}{2}} \\
    & =  \frac{2 \sqrt{2} C_1}{\sqrt{\pi}(D_{m^*}+1)} \frac{1}{\sqrt{qK}+\sqrt{qK+4}} e^{-K\frac{(D_{m^*}+1)}{2}}.
\end{align*}
\textit{(*)} is true since each term in the sum is positive, so, the sum is larger than one of them. \\
For all $\eta>0$, $\exists \Tilde{C}_\eta>0$, $\exists \Tilde{L}_\eta >0$ such that $\forall K > \Tilde{L}_\eta$, we have $\Tilde{C}_\eta e^{-\eta K} \leq \frac{1}{\sqrt{qK}+\sqrt{qK+4}}$. \\
So, 
\begin{align*}
    \forall \eta >0, \ \exists \Tilde{C}_\eta>0, \ &\exists \Tilde{L}_\eta >0, \ \forall K > \max\Big(\Tilde{L}_{C_1},\Tilde{L}_\eta\Big), \\
    &\text{FDR}(\Hat{m}(K)) \geq  \frac{2 \sqrt{2} C_1}{\sqrt{\pi}(D_{m^*}+1)} \Tilde{C}_\eta e^{-K\Big(\frac{D_{m^*}+1+2\eta}{2}\Big)},
\end{align*}
which gives~(\ref{LB_rate_K_infinity}) with $C_\eta = \frac{2 \sqrt{2} C_1}{\sqrt{\pi}(D_{m^*}+1)} \Tilde{C}_\eta$ and $L_\eta = \max\Big(\Tilde{L}_{C_1},\Tilde{L}_\eta\Big)$. \\

\vspace{0.3 cm}
Formula~(\ref{large_deviation}) automatically follows from~(\ref{CV_rate_K_infinity}) and~(\ref{LB_rate_K_infinity}). 
\qed

\subsection{General bounds.}
\label{proof_orthogonal}
\textbf{\textit{Proof of Corollary~\ref{corollary}}.}\\
By taking $u_j = X_j, \ \forall j \in \{1,\cdots,q\}$, then $\left(X_1,\cdots,X_q,u_{q+1},\cdots,u_n\right)$ is an orthonormal basis of $\mathbb{R}^n$. Consequently, $ \forall j \in \{1,\cdots,q\}, \ \langle X\beta^*,u_j\rangle = \langle X \beta^*, X_j \rangle = \beta^*_j$, which concludes the proof. \\
\qed

\section{Description of the simulation protocol}
\label{simulation_protocol}

In this section, we first describe the data simulation and the four considered scenarios. Then, some details of experimental protocol are provided.

\paragraph{Description of the data simulation.}
Given values of $n$ and $p$, we simulate $Y \sim \mathcal{N}(\beta^*,I_n)$ where $\beta^*$ is a vector satisfying $\beta^*_j \geq \beta^*_{j+1}$ for all $j \in \{1,\cdots,D_{m^*-1}\}$ to get ordered active variables. 
We consider four scenarios, described in Table~\ref{Table_scenario}, where values of $D_{m^*}$, $\beta^*$, $n$ and $\sigma^2$ vary and where the number of variables $p$ is always equal to $50$.

\begin{table}[!b]
\tabcolsep=9pt
\centering
\begin{tabular}{|m{1cm}||m{1.4cm}|m{4.5cm}|m{1.5cm}|m{1.1cm}|}
 \hline
 Scenario with $p=50$ & Active variable number & Non-zero coefficient amplitude in $\beta^*$ & Number of observations & Noise \newline amplitude \\
  \hline
  \hline
 (i) \newline Sparsity & $D_{m^*} \in \newline \{1,10,20\}$ & $\beta^*_{D_{m^*}} = 2$, \newline $\forall j \in \{1,\cdots,D_{m^*}-1\}$ \newline $\beta^*_j \sim \text{Unif}(\beta^*_{j+1}+0.5,\beta^*_{j+1}+1.5)$ & $n=50$ & $\sigma^2=1$ \\
  \hline
 (ii) \newline \hspace{-0.05 cm} Complexity & $D_{m^*} = 10$ & $\beta^*_{10} = 2$ with \newline $\forall j \in \{1,\cdots,9\}$ \newline $\beta^*_j \sim \text{Unif}(\beta^*_{j+1}+0.5,\beta^*_{j+1}+1.5)$ \newline \newline  $\beta^*_{10} = \frac{2}{10}$ with \newline  $\forall j \in \{1,\cdots,9\}$, \newline $\beta^*_j \sim \text{Unif}(\beta^*_{j+1}+0.05,\beta^*_{j+1}+0.15)$ \newline \newline  $\beta^*_{10} = 2$ with \newline $\forall j \in \{1,\cdots,9\}$ \newline $\beta^*_j \sim \text{Unif}(\beta^*_{j+1}+0.05,\beta^*_{j+1}+0.15)$ & $n=50$ & $\sigma^2=1$ \\
  \hline
 (iii) \newline High- \newline dimension & $D_{m^*} = 10$ & $\beta^*_{D_{m^*}} = 2$, \newline $\forall j \in \{1,\cdots,9\}$ \newline $\beta^*_j \sim \text{Unif}(\beta^*_{j+1}+0.5,\beta^*_{j+1}+1.5)$ & $n \in \newline \{30,50,300\}$ & $\sigma^2=1$ \\
  \hline
  (iv) \newline Noise & $D_{m^*} = 10$ & $\beta^*_{D_{m^*}} = 2$, \newline $\forall j \in \{1,\cdots,9\}$ \newline $\beta^*_j \sim \text{Unif}(\beta^*_{j+1}+0.5,\beta^*_{j+1}+1.5)$ & $n=50$ & $\sigma^2 \in \newline \{0.1,1,4\}$ \\
  \hline
\end{tabular}
\caption{Description of the four scenarios.
\label{Table_scenario}}
\end{table}

The scenario (i) allows us to evaluate the impact of the sparsity of the parameter $\beta^*$. The scenario (ii) allows us to evaluate how the values of the non-zero coefficients in $\beta^*$ complicate the identification of the active variables. In particular, the non-zero coefficients are close and, in the second configuration, some of them are smaller than the noise level $\sigma$. 
The scenario (iii) allows us to evaluate the behavior of our method in a high-dimensional context through the variation of the number of observations $n$, either smaller, equal or larger than the number of variables $p$. The last scenario (iv) allows us to evaluate the impact of the noise amplitude through different values of $\sigma^2$. \\
Note that for a fair comparison, the datasets where $n=30$ in scenario (iii) are inlcuded in those where $n=50$ which are included in those where $n=300$. 
Moreover, for the sake of reproducibility, the seed of the random number generator is identically fixed for each scenario. 

\paragraph{The \textit{toy data set}.}
We call the \textit{toy data set} the data set where $n=p=50$, $D_{m^*} = 10$, $\beta^*_{10} = 2$ and $\forall j \in \{1,\cdots,9\}$, $\beta^*_j \sim \text{Unif}(\beta^*_{j+1}+0.5,\beta^*_{j+1}+1.5)$. It corresponds to the reference data set in all scenarios.

\paragraph{Empirical estimations.}
For the empirical estimations, we simulate $\mathcal{D}$ a set of $1000$ data sets for each scenario. 
For each $d \in \mathcal{D}$ and for all $K>0$, the selected model $\Hat{m}^d(K)$ is obtained from $(Y^d,X^d)$. Since $m^*$ is known, the quantity $\text{FDP}(\Hat{m}^d(K))$ is calculable for each $d \in \mathcal{D}$ and the empirical estimator of $\text{FDR}(\Hat{m}(K))$ is the average of the $\text{FDP}(\Hat{m}^d(K))$. 
Concerning PR, we simulate $\Tilde{\mathcal{D}}$ a new set of $1000$ data sets for each scenario. New $\Tilde{Y}^d$ are generated on $\Tilde{\mathcal{D}}$, from the model~(\ref{regression}), and by using the $X^d$ on $\mathcal{D}$ to respect the fixed design assumption. The selected models $\Hat{m}^d(K)$ and the $\Hat{\beta}^d_{\Hat{m}(K)}$ estimators are extracted by solving~(\ref{our_model_selection}) from the training sets $(Y^d,X^d)$ on $\mathcal{D}$. The PR is evaluated from the validation sets $(\Tilde{Y}^d,X^d)$ on $\mathcal{\Tilde{D}}$ by the mean squared error~: 
\begin{equation}
\text{MSE}(\Hat{m}^d(K)) = \frac{1}{n} \underset{i=1}{\overset{n}{\sum}} \Big(\Tilde{Y}^d_{i} - \underset{j=1}{\overset{p}{\sum}} x^d_{ij} \Hat{\beta}_{\Hat{m}^d(K)_j}\Big)^2.
\label{PR_def}
\end{equation}
The empirical estimator of $\text{PR}(\Hat{m}(K))$ is the average of the $\text{MSE}(\Hat{m}^d(K))$. \\
To validate the quality of the empirical estimations, the central limit theorem is applied to get the $95\%$ asymptotic confidence intervals~: 
$$\Big[\text{FDR}(\Hat{m}(K))-1.96\frac{\Hat{\sigma}}{\sqrt{1000}},\text{FDR}(\Hat{m}(K))+1.96\frac{\Hat{\sigma}}{\sqrt{1000}}\Big]$$
and 
$$\Big[\text{PR}(\Hat{m}(K))-1.96\frac{\Hat{\sigma}}{\sqrt{1000}},\text{PR}(\Hat{m}(K))+1.96\frac{\Hat{\sigma}}{\sqrt{1000}}\Big],$$ 
where $\Hat{\sigma}$ is the unbiased empirical estimator of the standard deviation $\sigma$. Since their width do not exceed $0.011$ and $0.07$ for respectively the FDR and the PR, they are tight, meaning that the empirical estimations are close to the theoretical quantities $\text{FDR}(\Hat{m}(K))$ and $\text{PR}(\Hat{m}(K))$. \\ 

\newpage

\section*{Description of the supplementary file}
All the R scripts are available at~\url{https://github.com/PerrineLacroix/Trade_off_FDR_PR}.

Details about the estimation of the theoretical bounds of the FDR in Theorem~\ref{theorem}, graphs for the bounds $B(K,\Hat{\beta}_{\Hat{m}(\Tilde{K})},\Hat{\sigma}^2)$ applied on the $4$ scenarios described in Section~\ref{simulation_protocol}, as well as the study of the robustness of variable ordering, of the construction of random model collections and of the comparison of our algorithm with other variable selection procedures, are provided in a supplementary file. It is complementary to Section~\ref{trade_off_fixed_collection}. It is available in \cite{lacroix:hal-04625023}.


\section*{Acknowledgments}
This research is supported in part by a public grant as part of the Investissement d’avenir project, reference ANR-11-LABX-0056-LMH, LabEx LMH. IPS2 benefits from the support of the LabEx Saclay Plant Sciences-SPS (ANR-17-EUR-0007). \\
The authors warmly thank and are grateful to Pascal Massart (Laboratoire de Mathématiques d'Orsay, Université Paris-Saclay) for helpful discussions and valuable comments. 
The authors would like to sincerely thank the Editor and the two anonymous referees for their valuable comments, suggestions and feedbacks which improved the paper.

\newpage
\bibliographystyle{unsrt}  
\bibliography{reference}  

\end{document}